\documentclass[a4paper,11pt,reqno]{amsart}
\usepackage{latexsym,amscd,amssymb,url}
\usepackage{comment}
\usepackage{extarrows}
\usepackage[all,cmtip]{xy}  
\usepackage{graphicx}
\usepackage{xcolor}
\usepackage{enumerate} 
\definecolor{dblue}{rgb}{0,0,.6}
\usepackage[colorlinks=true, linkcolor=dblue, citecolor=dblue, filecolor = dblue, menucolor = dblue, urlcolor = dblue]{hyperref}

\makeatletter
\@namedef{subjclassname@2020}{%
  \textup{2020} Mathematics Subject Classification}
\makeatother

\numberwithin{equation}{section}

\newtheorem{theorem}{Theorem}[section]

\theoremstyle{plain}

\newtheorem{corollary}[theorem]{Corollary}

\newtheorem{definition}[theorem]{Definition}

\newtheorem{lemma}[theorem]{Lemma}

\newtheorem{proposition}[theorem]{Proposition}
\newtheorem{remark}[theorem]{Remark}

\newcommand{\del}{\partial}
\newcommand{\Z}{\mathbb Z}
\newcommand{\Q}{\mathbb Q}

\newcommand{\CP}{\mathbb P}

\newcommand{\im}{\operatorname{im}}

\newcommand{\Pic}{\operatorname{Pic}}

\newcommand{\id}{\operatorname{id}}

\newcommand{\Spec}{\operatorname{Spec}}

\newcommand{\CH}{\operatorname{CH}}

\newcommand{\sing}{\operatorname{sing}} 
\newcommand{\red}{\operatorname{red}} 
\newcommand{\Proj}{\operatorname{Proj}}

\newcommand{\Frac}{\operatorname{Frac}}

  \newcommand{\Tor}{\operatorname{Tor}}

  \newcommand{\coker}{\operatorname{coker}} 
  \newcommand{\spe}{\operatorname{sp}} 
    \newcommand{\longsquiggly}{\xymatrix{{}\ar@{~>}[r]&{}}}
       
\newcommand{\dashedlongrightarrow}{\xymatrix@1@=15pt{\ar@{-->}[r]&}}
\renewcommand{\longrightarrow}{\xymatrix@1@=15pt{\ar[r]&}}
\renewcommand{\mapsto}{\xymatrix@1@=15pt{\ar@{|->}[r]&}}
\renewcommand{\twoheadrightarrow}{\xymatrix@1@=15pt{\ar@{->>}[r]&}}
\newcommand{\hooklongrightarrow}{\xymatrix@1@=15pt{\ar@{^(->}[r]&}}
\newcommand{\congpf}{\xymatrix@1@=15pt{\ar[r]^-\sim&}}
\renewcommand{\cong}{\simeq}

\begin{document}    

\title[The diagonal of quartic fivefolds]{The diagonal of quartic fivefolds}

\author{Nebojsa Pavic} 
\address{Institute of Algebraic Geometry, Leibniz University Hannover, Welfengarten 1, 30167 Hannover , Germany.}
\email{pavic@math.uni-hannover.de}

\author{Stefan Schreieder} 
\address{Institute of Algebraic Geometry, Leibniz University Hannover, Welfengarten 1, 30167 Hannover , Germany.}
\email{schreieder@math.uni-hannover.de}

\date{January 10, 2023} 


\begin{abstract}    
We show that a very general quartic hypersurface in $\CP^6_k$ over a field  of characteristic different from 2 does not admit a decomposition of the diagonal, hence  is not retract rational. 
This generalizes a result of Nicaise--Ottem \cite{NO}, who showed stable irrationality over  fields of characteristic zero.
To prove our result, we introduce a new cycle-theoretic obstruction 
that may be seen as an analogue of 
the motivic obstruction for rationality  in characteristic zero, introduced by Nicaise--Shinder \cite{NS} and Kontsevich--Tschinkel \cite{KT}.
\end{abstract}

\maketitle

\section{Introduction}

A variety $X$ over a field $k$ is rational if it is isomorphic to $\CP^{\dim X}_k$ after removing proper closed subvarieties from both sides; it is stably rational if $X\times \CP^n_k$ is rational for some $n\geq 0$.
Moreover, $X$ is retract rational if there is an integer $N$ and rational maps $f:X\dashrightarrow \CP^N$ and $g:\CP^N\dashrightarrow X$ such that $g\circ f$ is defined and coincides with the identity: $g\circ f=\id_X$.
Finally,  $X$ is unirational if it receives a dominant rational map from some proejctive space.
We have the following well-known implications:
\begin{align} \label{eq:implication}
\text{rational}\ \ \ \Longrightarrow\ \ \ \text{stably rational}\ \ \ \  \Longrightarrow\ \ \ \text{retract rational}\ \ \ \  \Longrightarrow\ \ \ \text{unirational}.
\end{align}
By \cite{BCTSS} and \cite{artin-mumford}, the first and last implications are both strict over algebraically closed fields; it is an open problem whether this holds true for the second implication  as well.

Retract rational varieties admit a decomposition of the diagonal, which means that the point of the diagonal $\delta_X\in X_{k(X)}$ is rationally equivalent to $z_{k(X)}$ for some $k$-rational point $z\in X$.
In \cite{voisin}, with improvements in \cite{CTP,Sch-Duke}, Voisin used this implication to initiate a cycle-theoretic degeneration technique which, roughly speaking, allows to disprove retract rationality for varieties $X$ that admit a degeneration to a mildly singular variety $Y$ with a cohomological obstruction for the existence of a decomposition of the diagonal,  such as global differential forms \cite{totaro} or unramified cohomology \cite{CTP,HPT,Sch-JAMS}.
This method can be extended to cases where the special fibre of the degeneration breaks up into several pieces (see e.g.\ \cite[Lemma 2.4]{totaro}, \cite[Proposition 6.1]{Sch-torsion} and \cite[Theorem 8.5]{Sch-survey}).
However,  at least over algebraically closed fields, it is crucial that at least one component $Y_{i_0}$ of the special fibre $Y$ is irrational with some cohomological obstruction, while that obstruction class must vanish on the intersection of $Y_{i_0}$ with any other component $Y_j$ (e.g.\ if each component of $Y_{i_0}\cap Y_j$ is rational). 

Using the weak factorization theorem,  Nicaise--Shinder \cite{NS} and Kontsevich--Tschinkel \cite{KT} introduced powerful motivic obstructions for rationality and stable rationality in characteristic zero.
Their approach proves (stable) irrationality for varieties that admit degenerations to simple normal crossing varieties $Y=\bigcup_{i\in I}Y_i$, such that
$$
 [ \CP_k^{\dim Y}]+\sum_{\emptyset\neq J\subset I} (-1)^{|J|} [Y_{J}\times \CP^{|J|-1}_k]
$$
is nonzero in the free abelian group generated by (stable) birational equivalence classes of smooth projective $k$-varieties, where $Y_J=\bigcap_{j\in J}Y_j$.
For instance, it follows from this formula that a variety that specializes to a union of two smooth rational varieties that meet along a stably irrational variety is stably irrational. 
This approach has been implemented successfully by Nicaise--Ottem \cite{NO}, who used it to prove that a very general quartic fivefold over an  uncountable  algebraically closed field of characteristic zero is not stably rational, by writing down a degeneration into two components (that may be chosen to be rational) such that their intersection is stably irrational by the work of Hassett--Pirutka--Tschinkel \cite{HPT}.

The motivic method of Nicaise--Shinder and Kontsevich--Tschinkel does not seem to generalize to detect retract rationality. 
In particular, one may speculate that the motivic obstruction from \cite{NS,KT} could be a suitable tool to distinguish retract rational varieties from stably rational varieties and quartic fivefolds treated in \cite{NO} yield the first potential candidates, regarding the strictness of the second implication in (\ref{eq:implication}).

In this paper, we prove that quartic fivefolds are in fact not retract rational and hence do not give counterexamples to strictness of the second implication in (\ref{eq:implication}).

\begin{theorem}\label{thm:quartic-fivefold}
Let $k$ be an uncountable field of characteristic different from $2$.
A very general quartic $X\subset \CP^6_k$ does not admit a decomposition of the diagonal, hence is not retract rational.
\end{theorem} 

By a theorem of Morin (see \cite{conte-murre}), a general quartic fivefold $X\subset \CP^6_k$ as in the above theorem is known to be unirational.
On the other hand,  over fields of positive characteristic, even rationality of quartic fivefolds was previously open. 

The rationality problem for Fano hypersurfaces $X\subset \CP^{n+1}_k$ is a classical problem in algebraic geometry, see e.g.\ \cite{CG,IM,kollar,CTP,totaro,Sch-JAMS,Sch-torsion}.
In arbitrary dimension, the best known bound  is due to \cite{Sch-JAMS,Sch-torsion}, where it is shown that over uncountable  fields $k$, very general hypersurfaces of dimension $n\geq 3$ and degree $d\geq \log_2n +2$ (resp.\ $d\geq \log_2n+3$ if $\operatorname{char}(k)=2$) do not admit a decomposition of the diagonal and hence are neither retract rational, nor stably rational.
The case of cubic threefolds \cite{CG,Mur}, the aforementioned result of Nicaise--Ottem \cite{NO} and Theorem \ref{thm:quartic-fivefold} above are the only cases where this bound has been improved.
 
In order to use a similar degeneration in Theorem \ref{thm:quartic-fivefold}  as  Nicaise--Ottem used in  \cite{NO}, we roughly speaking face the problem of disproving the existence of a decomposition of the diagonal for the geometric generic fibre of a family that degenerates into a union of two rational components $Y=Y_1\cup Y_2$ such that $Y_{12}=Y_1\cap Y_2$ is integral and does not admit a decomposition of the diagonal. 
The naive idea is to perform a $2:1$ base change and to blow-up $Y_{12}$, to arrive at a semi-stable model whose special fibre has three components: two rational end components and a component in the middle that is a $\CP^1$-bundle over $Y_{12}$ and hence does not admit a decomposition of the diagonal. 

At this point, the following  problem arises: we could have started with the trivial family $\CP^n_R\to \Spec R$, blow-up a subvariety $Z$ of the special fibre followed by the blow-up of a general point of the exceptional divisor above $Z$.
If $Z$ does not admit a decomposition of the diagonal, then we arrive at a degeneration of $\CP^n$ into a chain of three components  
where again the two end components are rational, while the component in the middle does not admit a decomposition of the diagonal.
So how can we tell apart this degeneration, from the one discussed above? 

\subsection{Method}
We describe our solution to the above mentioned problem briefly in this section.
For this, let $R$ be a discrete valuation ring and let $\mathcal X\to \Spec R$ be a projective strictly semi-stable $R$-scheme with special fibre $Y=\bigcup_{i\in I} Y_i$.
There is a canonical map 
\begin{align*}
\Phi_{\mathcal X}:\CH_1(Y) \longrightarrow \ker\left(\deg:\bigoplus_{i\in I}\CH_0(Y_i)\to \Z\right),  \ \ \ \ \ 
 \gamma \mapsto \sum_{i\in I} \iota_i^\ast(\iota_\ast \gamma),
\end{align*}
where $\iota:Y\to \mathcal X$ and $\iota_i:Y_i\to \mathcal X$ denote the respective closed immersions, cf.\ Section \ref{subsec:obstruction} below.
It is easy to see that $\Phi_{\mathcal X}$ does not depend on $\mathcal X$, but in fact only on the special fibre $Y$. 
Indeed,  if $\gamma\in \CH_1(Y)$ is supported on $Y_{i_0}$, then  the contribution $\Phi_{\mathcal X,Y_i}(\gamma)\in \CH_0(Y_i)$ is given by
$$
\Phi_{\mathcal X,Y_i}(\gamma)=\begin{cases}
\gamma|_{Y_i\cap Y_{i_0}}\ \ & \text{if $i\neq i_0$,}\\
-\sum_{j\neq i_0}\gamma|_{Y_i\cap Y_{j}}\ \ \ \ &\text{if $i=i_0$,}
\end{cases}
$$
where $\gamma|_{Y_i\cap Y_{j}}$ is the zero-cycle on $Y_i$, given by the intersection of $\gamma$ (viewed as a $1$-cycle on $Y_i$) with   $Y_i\cap Y_{j}$ (viewed as a divisor on $Y_i$).
This description allows to compute $\Phi_{\mathcal X}$ effectively. 

If $A/R$ is an unramified extension of dvr's, then the base change $\mathcal X_A$ is again strictly semi-stable and we get a map $\Phi_{\mathcal X_A}$ as above.
Even if the residue field $\kappa$ of $R$ is algebraically closed, the residue field of $A$ may be a function field over $\kappa$ (e.g.\ $\kappa(Y_i)$ for some $i$) in which case $\Phi_{\mathcal X_A}$ is a map between  Chow groups of varieties over interesting non-closed fields.
In particular, even if $\Phi_{\mathcal X}$ is surjective,  this may very well fail for $\Phi_{\mathcal X_A}$.

\begin{theorem}\label{thm:method}
Let $R$ be a discrete valuation ring with algebraically closed residue field and let $\pi:\mathcal X\to \Spec R$ be a projective strictly semi-stable $R$-scheme.
\begin{enumerate}
\item If the generic fibre of $\pi$ admits a decomposition of the diagonal, then 
for any unramified extension $A/R$ of dvr's, the map $\Phi_{\mathcal X_{A}}$ is surjective.\label{item:method:1}
\item 
If the geometric generic fibre of $\pi$ admits a decomposition of the diagonal and the dual graph of the special fibre is a straight line (i.e.\ the components form a chain), then 
for any unramified extension $A/R$ of dvr's, the map $\Phi_{\mathcal X_{A}}$ is surjective modulo $2$. \label{item:method:2}
\end{enumerate} 
\end{theorem}

If the geometric generic fibre of $\pi$ admits a decomposition of the diagonal, then up to a base change and a modification of the total space, we arrive at a situation where the generic fibre admits such a decomposition, and so item (\ref{item:method:1}) may be applied.  
While this route could be taken to prove Theorem \ref{thm:quartic-fivefold},  it is significantly easier to work with the enhanced version in item (\ref{item:method:2}), where no additional base change and blow-ups are necessary, and so it allows us to work with a special fibre that has only few components. 

We like to think about Theorem \ref{thm:method} as a cycle-theoretic analogue of the motivic method from \cite{NS,KT} exploited in \cite{NO}.
Note however that there is no direct relation among the two methods and it may be possible to find situations where one applies, but not the other.

To give an idea of how to apply Theorem \ref{thm:method}, we state the following consequence as an example.

\begin{corollary} \label{cor:method}
Let $R$ be a dvr with algebraically closed residue field $k$ and let $\pi:\mathcal X\to \Spec R$ be a projective strictly semi-stable $R$-scheme whose special fibre $Y=Y_1\cup Y_2$ has two components.
Assume that
\begin{itemize}
\item $Y$ is universally $\CH_1$-trivial in the sense that for any field extension $L/k$, the natural map $\CH_1(Y)\to \CH_1(Y_L)$ is surjective;
\item $Y_{12}:=Y_1\cap Y_2$ is integral and its torsion order is even (cf.\ Section \ref{subsec:Tor(X)} below). 
\end{itemize}
Then the geometric generic fibre of $\pi$ does not admit a decomposition of the diagonal.
\end{corollary}

Theorem \ref{thm:quartic-fivefold} will be deduced from Theorem \ref{thm:method} and a degeneration that is inspired by   \cite{NO}.
Note however that checking that the condition in item (\ref{item:method:2}) of Theorem \ref{thm:method} is violated is subtle and requires several additional degenerations.
In particular, even though Corollary \ref{cor:method} illustrates well the basic idea, the technical details are more complicated and we are not able to literally apply the statement of Corollary \ref{cor:method}. 
The key input of course remains the striking example of Hassett--Pirutka--Tschinkel in \cite{HPT}. 

The problem of finding cycle-theoretic obstructions for rationality that are sensitive to semi-stable degenerations into unions of rational varieties has also been considered in \cite{BBGa,BBGb}.
Their approach relies on explicit identities in the prelog Chow ring and is different from ours.

\begin{remark} 
In \cite[Theorem 7.1]{NO}, Nicaise--Ottem showed that a very general complete intersection $X_{23}\subset \CP^6_k$ of type $(2,3)$  is not stably rational over any uncountable algebraically closed field $k$ of characteristic zero.  
Similarly to the case of quartic fivefolds, this result  generalizes as follows:  
for any uncountable field of characteristic different from $2$,  a very general complete intersection $X_{23}\subset \CP^6_k$ of type $(2,3)$ does not admit a decomposition of the diagonal and hence is not retract rational.
While this can also be deduced from Theorem \ref{thm:method}, the case of complete intersections of type $(2,3)$ in $\CP^6$ is in fact much easier than the case of quartic fivefolds treated in this paper. 
Indeed, by \cite[Proof of Theorem 7.1]{NO}, such complete intersections admit a degeneration into a union of two varieties meeting along a rational variety, such that one component does not admit a decomposition of the diagonal and so \cite[Lemma 2.4]{totaro} applies to give the result.
In \cite{skauli}, Skauli proves an enhancement of this result by actually constructing examples over $\Q$, where due to additional singularities in the special fibre slightly more careful arguments are necessary.
Note that these lines of arguments do not work for quartic fivefolds,  where the `obstruction for rationality' lies really in the intersection of the components of the degeneration, and not in the components  themselves.
\end{remark}

\section{Notation and conventions}

\subsection{Conventions} \label{subsec:convention}
All schemes are separated.
A variety is an integral scheme of finite type over a field. 
If $X$ is a scheme over a ring $R$ and $A/R$ is ring extension, then we write $X_A:=X\times_RA$.
A very general point of an irreducible scheme is a closed point outside a countable union of proper closed subsets.
For an integral scheme $X$ over a field $k$, we write $k(X)$ or $\kappa(X)$ for the function field of $X$; we use the latter whenever we prefer to make the ground field of $X$ not explicit in our notation.

For an abelian group $G$, we write $G/2$ instead of $G\otimes \Z / 2\Z$.
If $G_1\to G_2$ is a group homomorphism of abelian groups, then we denote by abuse of notation $G_2/G_1:= \coker (G_1\to G_2)$.

\subsection{Decomposition of the diagonal}
Let $X$ be a variety over a field $k$.
We denote by $\CH_i(X)$ the Chow group of algebraic cycles of $X$ of dimension $i$. 
One says that $X$ admits a decomposition of the diagonal, if
\begin{align} \label{eq:Delta}
[\Delta_X]=[z\times_k X]+ [Z]\in \CH_{\dim(X)}(X\times_k X),
\end{align}
where $\Delta_X\subset X\times_kX$ denotes the diagonal, $z$ is a zero-cycle on $X$ and $Z$ is a cycle on $X\times_k X$ which does not dominate the second factor.
Let now $k(X)$ denote the fraction field of $X$ and write $X_{k(X)}:=X\times_k k(X)$. We denote by
\[
\delta_{X}\in \CH_0( X_{k(X)} )
\]
the zero-cycle on $X_{k(X)}$ that is induced by pulling back $\Delta_X$ via $X_{k(X)}\to X\times_k X$. 
By the localization sequence, \cite[Proposition 1.8]{fulton} (\ref{eq:Delta}) is equivalent to $\delta_X=z_{k(X)}\in \CH_0( X_{k(X)} )$.

A proper variety $X$ over $k$ is said to have universally trivial Chow group of zero-cycles, if for any field extension $F\supset k$, the degree map $\deg :\CH_0(X_F)\to \mathbb{Z}$ is an isomorphism. If $X$ is in addition geometrically integral and smooth (over $k$), then $X$ has universally trivial Chow group of zero-cycles if and only if $X$ admits a decomposition of the diagonal (see \cite[Proposition 1.4]{CTP}).

\subsection{Torsion order} \label{subsec:Tor(X)}
The torsion order $\Tor(X)$ of a proper $k$-variety is the smallest positive integer $N$ such that $N\cdot \Delta_X$ admits a decomposition as in (\ref{eq:Delta}), or, equivalently, such that $N\cdot \delta_X=z_{k(X)}\in \CH_0(X_{k(X)})$ for some zero-cycle $z\in \CH_0(X)$.
The torsion order is $\infty$ if no such integer exists, cf.\ \cite{CL,Sch-torsion}.

\begin{lemma} \label{lem:torsion}
Let $X$ be a smooth projective variety over an algebraically closed field $k$.
Assume that the torsion order of $X$ is finite and divisible by $e\geq 2$.
Then for any $z\in \CH_0(X)$, the following holds in $\CH_0(X_{k(X)})$:
$$
\delta_X-z_{k(X)}\not \equiv 0\mod e . 
$$
\end{lemma}
\begin{proof}
For a contradiction, assume that $\delta_X \equiv z_{k(X)} \mod e $.
Then there is a zero-cycle $\epsilon\in \CH_0(X_{k(X)})$ with $ \delta_X=z_{k(X)}+e\epsilon$.
By assumptions, there is also a positive integer $d$ such that $\Tor(X)=d\cdot e$. 
Hence, $ed\Delta_X=z'\times X+R$ for some zero-cycle $z'$ on $X$ and some cycle $R$ on $X\times X$ that does not dominate the second factor.
Base changing everything to $k(X)$, we get a similar decomposition of $ed\cdot \Delta_{X_{k(X)}}$.
Since $\Delta_{X_{k(X)}}$ acts as the identity on $\CH_0(X_{k(X)})$, we find that 
$$
ed\cdot \epsilon=ed\cdot \Delta_{X_{k(X)}}^\ast(\epsilon)= \deg(\epsilon)\cdot z'_{k(X)}.
$$
Hence,
$$
d \delta_X=d z_{k(X)}+ed\cdot \epsilon =z''_{k(X)}.
$$
for some zero-cycle $z''\in \CH_0(X)$.
This implies that $\Tor(X)$ divides $d$ and so $e=1$, which contradicts our assumptions.
This proves the lemma.
\end{proof}

\subsection{Chains of divisors} \label{subsec:chain}
We call a closed and reduced subscheme $D=\bigcup^n_{i=1} D_i$ of a scheme $X$ with irreducible components $D_i$ of pure codimension one a chain of divisors, if the scheme-theoretic intersections $D_{i-1}\cap D_i$ and $D_{i}\cap D_{i+1}$ are disjoint from each other in $D_i$ for all $1< i < n$ and if all the other intersections $D_i\cap D_j$ with $i\neq j$ are empty. 
We call $D=\bigcup D_i$ a chain of Cartier divisors, if the $D_i$'s are in addition Cartier in $X$.

\subsection{Semi-stable models}
Let $R$ be a discrete valuation ring with residue field $k$ and fraction field $K$. 
A proper flat $R$-scheme $\mathcal X\to \Spec R$ is called strictly semi-stable, 
if the special fibre $Y=\mathcal{X}\times_R k$ is a geometrically reduced simple normal crossing divisor on $\mathcal{X}$. 
In other words, the components of $Y$ are smooth Cartier divisors and the intersection of $r$ different components is either empty, or smooth of codimension $r$.  
The total space $\mathcal X$ of a strictly semi-stable $R$-scheme is automatically regular, because the special fibre is contained in the regular locus by assumption and $\mathcal X$ is proper over $R$, so that any point of the generic fibre specializes to a point of the special fibre.
The special fibre $Y$ is called a chain of Cartier divisors if $Y\subset \mathcal X$ is a chain of Cartier divisors in the sense of Section \ref{subsec:chain}.

\section{Chow-theoretic obstruction:  the generic fibre}

\subsection{The obstruction map}\label{subsec:obstruction}

\begin{definition}  \label{definition:Phi}
Let $R$ be a discrete valuation ring and let $\mathcal X\to \Spec R$ be a strictly semi-stable $R$-scheme with special fibre $Y$.
Let $Y_i$ with $i\in I$ be the irreducible components of $Y$ and let $\iota:Y\to \mathcal X$ and $\iota_i:Y_{i}\to \mathcal X$ denote the natural embeddings. 
We define $\Phi_{\mathcal X,Y_i}:\CH_{1}(Y)\longrightarrow  \CH_{0}(Y_{i}) $ to be the composition
\begin{align}  \label{def:Phi}
\Phi_{\mathcal X,Y_i} : \CH_1(Y)\stackrel{\iota_\ast}\longrightarrow\CH_1(\mathcal X) \stackrel{ \iota_i^\ast}\longrightarrow  \CH_{0}(Y_{i}) 
\end{align}
and we denote by $\Phi_{\mathcal X}$ the direct sum
\begin{align} \label{def:Phi_sum}
\Phi_{\mathcal X}:=\sum_{i\in I}\Phi_{\mathcal X,Y_i} :\CH_1(Y)\longrightarrow \bigoplus_{i\in I} \CH_{0}(Y_{i}) .
\end{align}
\end{definition}

We have the following simple but useful lemma, which shows that  $\Phi_{\mathcal X}$ depends only on the special fibre $Y$ of the $R$-scheme $\mathcal X$.

\begin{lemma}\label{lem:Phi} 
In the notation of Definition \ref{definition:Phi}, let $Y_{ij}:=Y_{i}\cap Y_j$, denote by $\iota_{ij}:Y_{ij} \to Y_{j}$ and $\iota_i:Y_i\to Y$ the natural inclusions, and write $\gamma_i|_{Y_{ji}}:=\iota_{ji}^\ast \gamma_i$ for  $\gamma_i\in \CH_1(Y_i)$.
\begin{enumerate}
\item For any $\gamma_i\in \CH_1(Y_i)$, we have
\begin{align*} 
 \Phi_{\mathcal X,Y_j}((\iota_{i})_\ast\gamma_i)=
 \begin{cases} (\iota_{ij})_\ast (\gamma_i|_{Y_{ji}}) \in  \CH_{0}(Y_{j}) ,\ \ &\text{for $j\neq i$},\\
 -\sum_{\substack{k\in I,\\ k\neq i}} (\iota_{ki})_\ast ( \gamma_i|_{Y_{ki}} ) \in  \CH_{0}(Y_{i}),  \ \ &\text{for $j= i$}.
 \end{cases}
\end{align*}

\item Let $\gamma=\sum_{i\in I}(\iota_i)_\ast \gamma_i\in \CH_1(Y)$.
Then
$$
\Phi_{\mathcal X,Y_i} (\gamma ) =\sum_{j\in I\setminus \{i\}} (\iota_{ji})_\ast \gamma_{j} |_{Y_{ji}} - \sum_{j\in I\setminus \{i\}}  (\iota_{ji})_\ast \gamma_{i} |_{Y_{ji}}  \in \CH_{0}(Y_{i}) .
$$
\end{enumerate} 
\end{lemma}
\begin{proof}
Note that the restrictions $\iota_{j i}^\ast$ are well-defined, because $Y$ is a simple normal crossing divisor on $\mathcal{X}$.
The first item follows then in the case $j\neq i$ directly from the definition of the intersection product (see \cite[Theorem 6.2 (a)]{fulton}), while the case $j=i$ follows from the fact that $[Y_i]=- \sum_{k\neq i}[Y_k]$ in $\Pic(\mathcal{X})$.
The second item is a direct consequence of the first.
\end{proof}

\subsection{Obstructing decompositions on the generic fibre}
The homomorphism $\Phi_{\mathcal X}$ from Definition \ref{definition:Phi} will be our main tool to obstruct the existence of decompositions of the diagonal in this paper.
This rests on two observations.
Firstly, if $\gamma\in \CH_1(Y)$, then 
$$
\deg(\Phi_{\mathcal X}(\gamma))=\deg \left( \sum_{i\in I} \iota_i^\ast \iota_\ast \gamma \right)=\deg(\iota^\ast \iota_\ast \gamma)=0
$$
and so the image of $\Phi_{\mathcal X}$ is always contained in the kernel of the degree map
$$
\deg:\bigoplus_{i\in I} \CH_{0}(Y_{i})\longrightarrow \Z,\ \ (z_i)_{i\in I}\mapsto \sum_{i\in I}\deg(z_i) .
$$
Secondly, whenever $A/R$ is an unramified extension of dvr's, then $\mathcal X_A:=\mathcal X\times_RA$ is a strictly semi-stable $A$-scheme.
In particular, if $L$ denotes the residue field of $A$ (which will be an extension of the residue field  of $R$), then we get from Definition \ref{definition:Phi} a homomorphism
\begin{align} \label{eq:Phi_X_A}
\Phi_{\mathcal X_A} :\CH_1(Y_L)\longrightarrow \ker\left( \deg: \bigoplus_{i\in I} \CH_{0}(Y_{i,L})\to \Z \right)  .
\end{align} 
Item (\ref{item:method:1}) in Theorem \ref{thm:method} follows from the following result.

\begin{proposition} \label{prop:Phi}
Let $R$ be a discrete valuation ring with residue field $k$ and fraction field $K$.
Let $\mathcal{X}\to \Spec R$ be a strictly semi-stable projective $R$-scheme whose generic fibre $X$ admits a decomposition of the diagonal.
Let $Y_i$ with $i\in I$ be the components of the special fibre $Y$.
Then for any unramified extension $A/R$ of dvr's,  the homomorphism in  (\ref{eq:Phi_X_A}) is surjective. 
 \end{proposition}
 
\begin{proof}  
By Lemma \ref{lem:Phi},
 $\Phi_{\mathcal X_A}$ depends only on the special fibre $Y_L$ and hence it remains the same after replacing $A$ by its completion. 
In particular, we may assume that $A$ is complete.
Let $(z_i)_{i\in I}\in \oplus_{i\in I} \CH_{0}(Y_{i,L})$ be a collection of zero-cycles with $\sum_{i}\deg(z_i)=0$.
By the moving lemma (see e.g.\ \cite{Roberts} or \cite[Theorem 2.13]{levine}), we may assume that the support of each $z_i$ is contained in the smooth locus of $Y_L$.

Combining inflation of local rings (see Lemma \ref{lem:inflation-of-local-rings})  with Hensel's lemma  (see e.g. \ \cite[Theorem 18.5.17]{EGAIV.4}), it follows that any zero-cycle supported on the smooth locus of $Y_L$ lifts to a one-cycle on $\mathcal X_A$.
In particular, there is a one-cycle $z\in \CH_1(\mathcal X_A)$ that is flat over $A$ and such that
$$
z\cap Y_{i,L}=z_i
$$
for all $i\in I$.

The restriction of $z$ to the generic fibre of $\mathcal X_A\to \Spec A$ is then a zero-cycle of degree zero and so it is rationally equivalent to zero, because $X$ has universally trivial Chow group of zero-cycles by assumption.
The restriction map
$
\CH_1(\mathcal X_{A})\longrightarrow \CH_0(  X_{\Frac A})
$ fits into the following localization exact sequence (see \cite[Section 4.4]{fulton2}): 
$$
\CH_1(Y_{L})\stackrel{\iota_\ast}\longrightarrow \CH_1(\mathcal X_{A})\longrightarrow \CH_0(  X_{\Frac A})\longrightarrow 0
$$
and so we find that there is a one-cycle $\gamma\in \CH_1(Y_{L})$ with
\begin{align} \label{eq:gamma}
z= \iota_\ast \gamma \in \CH_1(\mathcal X_{A}).
\end{align}
Since $z\cap Y_{i,L}=z_i$ for all $i\in I$, this implies that
$$
\Phi_{\mathcal X_A}(\gamma)=(z_i)_{i\in I} \in \bigoplus_{i\in I} \CH_{0}(Y_{i,L}).
$$
Hence, $\Phi_{\mathcal X_A}$ in (\ref{eq:Phi_X_A}) is surjective, as we want. 
\end{proof}

\section{Chow-theoretic obstruction:  the geometric generic fibre}

The purpose of this section is to prove item (\ref{item:method:2}) of Theorem \ref{thm:method}.
This yields an obstruction to decompositions of the diagonal of a smooth projective variety which specializes to a chain of smooth varieties such as the union $Y=Y_1\cup Y_2$ of two smooth varieties meeting along a smooth irreducible divisor $Y_1\cap Y_2$.
We will apply this obstruction later to a degeneration of quartic fivefolds that is similar to the one in \cite{NO}.
The precise statement of our obstruction is as follows.

\begin{theorem}\label{thm:Phi-geom-gen} 
Let $R$ be a discrete valuation ring with algebraically closed residue field and let $\mathcal{X}\to\Spec R$ be a strictly semi-stable projective $R$-scheme whose special fibre $Y=\bigcup_{i\in I} Y_i$ is a chain of Cartier divisors. 
Assume that the geometric generic fibre of $\mathcal{X}\to \Spec R$ has a decomposition of the diagonal.
Then for any unramified extension $A/R$ of dvr's,  with induced extension $L/k$ of residue fields, the natural map
$$
 \Phi_{\mathcal X_A} :\CH_1(Y_L)/2\longrightarrow \ker\left( \deg: \bigoplus_{i\in I} \CH_{0}(Y_{i,L})/2\to \Z/2 \right) ,
$$
given by reduction modulo two of (\ref{eq:Phi_X_A}), is surjective. 
\end{theorem}

In applications it will be useful to note that by inflation of local rings, an unramified extension $A/R$ of dvr's exists for any given extension $L/k$ of residue fields.

\begin{lemma} \label{lem:inflation-of-local-rings}
Let $R$ be a discrete valuation ring with residue field $k$.
For any field extension $L/k$, there is an unramified extension of dvr's  $A/R$ which   induces $L/k$ on the residue fields.
\end{lemma}
\begin{proof}
By inflation of local rings, see 
\cite[Chapter IX,  Appendice,  Corollaire du Th\'eor\`eme 1]{bourbaki},
there is a flat local $R$-algebra $A$ such that $A$ is a discrete valuation ring with $\mathfrak m_A = \mathfrak m_R R$ and $L=A/\mathfrak m_A$.
In particular, $A/R$ is an unramified extension of dvr's which induces $L/k$ on the residue fields, as we want.
\end{proof}

Theorem \ref{thm:Phi-geom-gen} will follow from Proposition \ref{prop:Phi} together with a careful analysis of the effect of ramified base changes $R \subset \tilde{R}$ of discrete valuation rings.

\subsection{A preliminary lemma}
The following simple lemma will be needed in the proof of Theorem \ref{thm:Phi-geom-gen}.

\begin{lemma}\label{lem:CH_1Y}
Let $\mathcal{X}\to\Spec R$ be a strictly semi-stable $R$-scheme, such that the special fibre $Y=\bigcup_{i=1}^n Y_i$ of $\mathcal{X}$ is a chain of Cartier divisors.
Let $Y_{i,i+1}:=Y_i\cap Y_{i+1}$ 
 and let $A/R$ be an unramified extension of dvr's with induced extension $L/k$ of residue fields.
Then we have the following, where
on the right hand side of \eqref{eq:sum-cycles}, \eqref{eq:CH_1Y-2}, \eqref{eq:CH_1Y} and \eqref{eq:Phi-computation1} below, we leave out the respective pushforward maps by the natural inclusions  by slight abuse of notation.
\begin{enumerate}[(a)] 
\item \label{item:lem:CH_1Y-1} Let 
\begin{equation}\label{eq:sum-cycles}
\gamma=\gamma_{1}+\ldots + \gamma_{n}\in \CH_1( Y_L ) ,
\end{equation}
such that $\gamma_{i}$ is in $\CH_1 (Y_{i , L})$. 
Then for all $i$, we have 
\begin{equation}\label{eq:CH_1Y-2}
\Phi_{\mathcal X_A,Y_{i,L}} (\gamma ) = \gamma_{i-1} |_{Y_{i-1 , i , L}} - \gamma_{i} |_{Y_{i-1, i, L}} - \gamma_{i} |_{Y_{i , i+1, L}} + \gamma_{i+1} |_{Y_{i, i+1, L}} \in \CH_{0}(Y_{i,L}).
\end{equation} 
\item \label{item:lem:CH_1Y-2}
  If, additionally, for some $0 < i < n$ the intersections $Y_{i-1 ,i}$ and $Y_{i,i+1}$ are isomorphic and if $Y_i$ is a $\CP^1_k$-bundle over $Y_{i-1,i}\cong Y_{i, i+1}$ with projection $q_{i}:Y_i \to Y_{i-1 , i}$, then any $\gamma\in \CH_l (Y_L)$ is of the form
\begin{equation}\label{eq:CH_1Y}
\gamma = \gamma_{1}+\ldots + \gamma_{{i-1}} + q^\ast_{i}\alpha_{i} + \gamma_{{i+1}} +\ldots + \gamma_{n } \in \CH_l (Y_L)
\end{equation}
for some $\alpha_{i} \in \CH_{0}(Y_{i-1, i, L})$ and $\gamma_j\in \CH_{1}(Y_{j, L})$, and we have that
\begin{equation}\label{eq:Phi-computation1}
\Phi_{\mathcal X_A,Y_{i,L}} ( \gamma ) = \gamma_{{i-1}} |_{Y_{i-1, i, L}} - 2\alpha_{i} + \gamma_{{i+1}} |_{Y_{i, i+1, L}} \in \CH_0(Y_{i,L}).
\end{equation} 
\end{enumerate}
\end{lemma} 

\begin{proof}
Item (\ref{item:lem:CH_1Y-1}) follows directly from Lemma \ref{lem:Phi}.
Moreover, (\ref{eq:Phi-computation1}) follows from (\ref{eq:CH_1Y})  and item (\ref{item:lem:CH_1Y-1}) and so it suffices to prove (\ref{eq:CH_1Y}).
That is, we need to show that $\CH_1(Y_L)$ is generated by 
\begin{equation}\label{eq:genCH1(Y0)}
\CH_1(Y_{1,L})\oplus\ldots \oplus \CH_1( Y_{i-1,L})\oplus q^\ast_{i}\CH_{0}(Y_{i,i-1, L})\oplus \CH_1(Y_{i+1, L})\oplus\ldots\oplus \CH_1(Y_{n,L}) .
\end{equation}

As $Y_{i,L}$ is a $\CP_k^1$-bundle over $Y_{i-1,i,L}$ and as $\iota_{i-1, i} : Y_{i,i-1,L}\hookrightarrow Y_{i, L}$ is a section of $q_{Y_{i}}: Y_{i, L}\to Y_{i-1,i,L}$, we can write any $x$ in $\CH_1(Y_{i,L})$ as 
\[
x= q_{{i}}^\ast (x_0 ) + (\iota_{i-1,i})_\ast (y)
\]
with $x_0$ in $\CH_{0}(Y_{i-1,i,L})$ and $y$ in $\CH_1(Y_{i-1,i,L})$ (see e.g.\ \cite[Theorem 3.3 (b)]{fulton}). 
It is clear, however, that $ (\iota_{i-1,i})_\ast (y)$ in $\CH_1(\tilde{Y}_L )$ is an element which comes from $\CH_1(Y_{i-1,L})$.
This proves \eqref{eq:genCH1(Y0)}, which concludes the lemma. 
\end{proof}

\subsection{Analyses of base change}

Let $\tilde{R}/ R$ be a finite (possibly ramified) extension of discrete valuation rings and let $\mathcal X\to \Spec R$ be a strictly semi-stable $R$-scheme.
We consider the base change $\mathcal{X}_{\tilde{R}}:=\mathcal{X}\times_R\tilde{R}$. 
Following Hartl (Proof of \cite[Proposition 2.2]{hartl}), there is a finite sequence of blow-ups
\[
\tilde{\mathcal{X}}:=\mathcal{X}_r\to\mathcal{X}_{r-1}\to \ldots\to\mathcal{X}_1\to\mathcal{X}_{\tilde{R}} ,
\]
where each step $\mathcal{X}_{i+1}\to\mathcal{X}_i$ is the consecutive blow-up of (strict transforms of) all irreducible components of $Y$ that are not Cartier, such that 
\begin{equation}\label{eq:base-change-ramified-extension}
\tilde{\mathcal{X}}\to \Spec\tilde{R} 
\end{equation}
is strictly semi-stable. 
In particular, 
the special fibre $\tilde{Y}$ of $\tilde{\mathcal{X}}$ is given by
\begin{equation}\label{eq:base-change-ramified-extension-special-fibre}
\tilde{Y} = Y_1\cup\bigcup_{j=1}^r R_{1,j}\cup Y_2 \cup\bigcup_{j=1}^r R_{2,j}\cup\ldots\cup Y_n .
\end{equation}
Note that $\tilde{Y}$ is again a chain of non-singular Cartier divisors and there is a $\CP^1_k$-bundle structure of $R_{i,j}$ over $Y_{i,i+1}$ for all $j>0$ and all $i$, and we denote the corresponding projection by $q_{R_{i,j}}:R_{i,j}\to Y_{i,i+1}$. 
Moreover, the intersections $Y_i\cap R_{i , 1}$ and $R_{i,j}\cap R_{i,j+1}$ are isomorphic to $Y_{i,i+1}$, for all $1\leq i< n$ and all $1\leq j \leq r$.

\begin{proposition}\label{prop:Phi-and-base-change}
Let $R$ be a discrete valuation ring with algebraically closed residue field $k$.
Let $\mathcal{X}\to\Spec R$ be a strictly semi-stable $R$-scheme such that the special fibre $Y=\bigcup_{i=1}^n Y_i$ of $\mathcal{X}$ is a chain of Cartier divisors. 
Let $R\subset \tilde{R}$ be a finite extension of discrete valuation rings and let $\tilde{\mathcal{X}}\to\Spec \tilde{R}$ be as in \eqref{eq:base-change-ramified-extension} with special fibre $\tilde{Y}$ as in \eqref{eq:base-change-ramified-extension-special-fibre}.
Let $A/R$ be an unramified extension of dvr's with induced extension $L/k$ of residue fields and denote by $\tilde{\mathcal X}_{\tilde A}$ the base change of $\tilde{\mathcal{X}}$ to an unramified extension $\tilde A$ of $\tilde R$ which induces $L/k$ on the residue fields ($\tilde A$ exists by Lemma \ref{lem:inflation-of-local-rings}).
Then the following holds:
If
$\Phi_{\widetilde{\mathcal X}_{\tilde A}}$ is surjective modulo $2$, then $\Phi_{ \mathcal X_A}$ is surjective modulo $2$. 
\end{proposition}

\begin{proof}
We will split the proof of the proposition into two cases, depending on the parity of the integer $r$ appearing in  \eqref{eq:base-change-ramified-extension-special-fibre}.
Before we do so, let us first make some general observations, which will be used in both cases. 

By Lemma \ref{lem:CH_1Y} (\ref{item:lem:CH_1Y-2}) \eqref{eq:CH_1Y}, any $\gamma \in\CH_1(\tilde{Y}_L)$ can be written as 
\[
\gamma=\gamma_{Y_1}+\sum_{j=1}^r q_{R_{1,j}}^\ast\alpha_{R_{1,j}} + \gamma_{Y_2} +\sum_{j=1}^r q^\ast_{R_{2,j}}\alpha_{R_{2,j}}+\ldots + \gamma_{Y_n} ,
\]
for some $\gamma_{Y_i}\in \CH_1(Y_{i,L})$ and $\alpha_{R_{i,j}}\in\CH_{0}(Y_{i,i+1,L})$, where $Y_{i,i+1,L}:=Y_{i,L}\cap Y_{i+1,L}$. 
We then have that 
\[
  \Phi_{\tilde{\mathcal X}_{\tilde A},R_{i,j,L}}(\gamma )\equiv \alpha_{R_{i,j-1}}- 2\alpha_{R_{i,j}} +  \alpha_{R_{i,j+1}}\equiv \alpha_{R_{i,j-1}} + \alpha_{R_{i,j+1}}\bmod 2
\]
in $\CH_{0}(R_{i, j, L})/2$ for all $1< j < r$ and all $i$, where in the first equality we use Lemma \ref{lem:CH_1Y} (\ref{item:lem:CH_1Y-2}) \eqref{eq:Phi-computation1}. 
(Here and in what follows we neglect, by slight abuse of notation, the respective pushforwards to $R_{i, j, L}$ and to $Y_{i,L}$ in our notation, whenever no confusion is likely to arise.)
Moreover, we have that 
\begin{align*}
 \Phi_{\tilde{\mathcal X}_{\tilde A},R_{i,1,L}}(\gamma) \equiv  \gamma_{Y_i} |_{Y_{i,i+1,L}} + \alpha_{R_{i,2}} \mod 2 , \\
 \Phi_{\tilde{\mathcal X}_{\tilde A},R_{i,r,L}} (\gamma) \equiv \alpha_{R_{i,r-1}}+ \gamma_{Y_{i+1}}|_{Y_{i,i+1,L}} \mod 2 ,
\end{align*}
by  Lemma \ref{lem:CH_1Y} (\ref{item:lem:CH_1Y-1}).

Let us now assume that  $\Phi_{\tilde{\mathcal X}_{\tilde A},R_{i,j,L}}(\gamma )\equiv 0\bmod 2$, for all $j> 0$ and all $i$.
Pushing the above equations forward via $q_{R_{i,j}}:R_{i,j,L}\to Y_{i,i+1,L}$ for $j>0$, we then obtain
\begin{equation}\label{eq:alpha_R_j}
\begin{aligned} 
\gamma_{Y_i} |_{Y_{i,i+1,L}} 
&\equiv \alpha_{R_{i,2}}\equiv\alpha_{R_{i,4}}\equiv\ldots\equiv\alpha_{R_{i, 2\lfloor r/2 \rfloor}}\mod 2 , \\
\alpha_{R_{i,r-2\lfloor r/2 \rfloor+1}}
&\equiv\ldots\equiv\alpha_{R_{i,r-3}}\equiv\alpha_{R_{{i,r-1}}}\equiv \gamma_{Y_{i+1}}|_{Y_{i,i+1,L}}\mod 2, \\
\alpha_{R_{i,1}}
&\equiv\alpha_{R_{i,3}}\equiv\ldots \equiv \alpha_{R_{{i,2\lceil r/2 \rceil-1}}}\mod 2 ,
 \end{aligned}
\end{equation}
in $\CH_0(Y_{i,i+1,L})/2$. 

\textbf{Case 1.} The integer $r$  appearing in  \eqref{eq:base-change-ramified-extension-special-fibre} is even.

In this case we will prove that for any $\gamma\in \CH_1(\tilde{Y}_L)$, such that $\Phi_{\tilde{\mathcal X}_{\tilde A},R_{i,j,L}}(\gamma )\equiv 0 \bmod 2$ for all $j$ and all $i$, we have that
\begin{equation}\label{eq:commutativity-mod-2}
\Phi_{\mathcal X_A} (q_\ast (\gamma ))=q_\ast \Phi_{\tilde{\mathcal X}_{\tilde A}}( \gamma )\ 
 \in\bigoplus_i\CH_{0}(Y_{i,L})/2 ,
\end{equation}
where $q:\tilde{Y}\to Y$ denotes the morphism induced by $\tilde{\mathcal X}\to \mathcal X$. 
This clearly implies that  if $\Phi_{\widetilde{\mathcal X}_{\tilde A}}$ is surjective modulo $2$, then $\Phi_{ \mathcal X_A}$ is surjective modulo $2$.

To prove \eqref{eq:commutativity-mod-2},  note first that since  $r$ is even,   \eqref{eq:alpha_R_j} reads as
\begin{equation}\label{eq:alpha_R_j_even}
\begin{aligned} 
\gamma_{Y_i} |_{Y_{i,i+1,L}} 
&\equiv \alpha_{R_{i,2}}\equiv\alpha_{R_{i,4}}\equiv\ldots\equiv\alpha_{R_{i, r }}\mod 2 , \\
\alpha_{R_{i,1}}
&\equiv\alpha_{R_{i,3}}\equiv\ldots\equiv\alpha_{R_{{i,r-1}}}\equiv \gamma_{Y_{i+1}}|_{Y_{i,i+1,L}}\mod 2 
 \end{aligned}
\end{equation}
in $\CH_0(Y_{i, i+1,L})/2$. Furthermore, we note that the morphism $q : \tilde{Y}\to Y$ is the identity morphism onto its image when restricted to the components $Y_i$ and is the projection $q_{R_{i,j}}:R_{i,j}\to Y_{i,i+1}$ when restricted to the components $R_{i,j}$. 
This, together with the assumption $\Phi_{ {\tilde{\mathcal X}}_{\tilde{A}},R_{i,j,L}} (\gamma )\equiv 0\bmod 2$ for all $j$ and all $i$, implies that 
\begin{align*}
q_\ast \Phi_{\tilde{\mathcal X}_{\tilde A}}(\gamma ) 
& = \sum_{i=1}^n   \alpha_{R_{i-1,r}} -  \gamma_{Y_i}|_{Y_{i-1 ,i,L}} - \gamma_{Y_i}|_{Y_{i ,i+1 ,L}} + \alpha_{R_{i,1}} \in \bigoplus_i \CH_0 ( Y_{i,L}) /2 ,
\end{align*}
where we used Lemma \ref{lem:CH_1Y} \eqref{item:lem:CH_1Y-1}.
(Here, $\alpha_{R_{0, r}}$, $\gamma_{Y_1}|_{Y_{0, 1, L}}$, $\gamma_{Y_n}|_{Y_{n, n+1, L}}$ and $\alpha_{R_{n, 1}}$ are defined to be $0$.)

Using then the relations in \eqref{eq:alpha_R_j_even}, we can further rewrite this equation as
\begin{align*}
 q_\ast \Phi_{\tilde{\mathcal X}_{\tilde A}}(\gamma ) 
&\equiv  \sum^n_{i=1}  \gamma_{Y_{i-1}}|_{Y_{i-1,i,L}} - \gamma_{Y_i}|_{Y_{i-1 ,i,L}} -  \gamma_{Y_i}|_{Y_{i ,i+1 ,L}} +  \gamma_{Y_{i+1}}|_{Y_{i,i+1,L}} \\
&\equiv \sum^n_{i=1} \Phi_{\mathcal X_A} ( \gamma_{Y_i})\equiv \Phi_{\mathcal X_A}  ( q_* \gamma )\mod 2 ,
\end{align*}
where in the second equation we used Lemma \ref{lem:CH_1Y} \eqref{item:lem:CH_1Y-1} again ($\gamma_{Y_0}|_{Y_{0, 1, L}}$ and $\gamma_{Y_{n+1}}|_{Y_{n, n+1, L}}$ are set to be $0$ here). 
This proves \eqref{eq:commutativity-mod-2} and hence concludes the proof of the proposition in the case where $r$ is even.

\textbf{Case 2.} The integer $r$  appearing in  \eqref{eq:base-change-ramified-extension-special-fibre} is odd.

We assume that  $\Phi_{\widetilde{\mathcal X}_{\tilde A}}$ is surjective modulo $2$ and that $r$ is odd.
We then need to show that $\Phi_{ {\mathcal X}_{ A}}$ is surjective modulo $2$.
A simple formula analogous to \eqref{eq:commutativity-mod-2} does not seem to hold in this case. 
Instead, our argument relies on the claim that, since $r$ is odd and $\Phi_{\widetilde{\mathcal X}_{\tilde A}}$ is surjective modulo $2$, we have that
\begin{align} \label{eq:surj-mod-2}
\Psi : \bigoplus_{i=1}^n \CH_1(Y_{i,L}) \longrightarrow \bigoplus_{i=1}^{n-1}\CH_0(Y_{i,i+1,L}),\ \ \ \sum_i \gamma_i \mapsto \sum_i \gamma_i |_{Y_{i,i+1,L}}+\gamma_{i+1}|_{Y_{i,i+1,L}}
\end{align}
is surjective modulo 2.
Indeed, let $\alpha' = \alpha'_1 + \ldots + \alpha'_{n-1}$ be in $\bigoplus_i \CH_0 (Y_{i , i+1, L})$ with $\alpha'_i $ in $\CH_0 (Y_{i, i+1, L})$.
Denote by $\iota_{{i }} : Y_{i , i+1 , L} \to Y_{i, L }$ the natural embedding and  let $\iota_{i , j } : Y_{i, i+1 , L} \to R_{i, j , L} $ denote the embeddings $Y_{i , L} \cap R_{i , 1, L} \subset R_{i , 1 , L}$ for $j=1$ and $R_{i , j-1 , L} \cap R_{i , j, L} \subset R_{i , j , L}$ for $1 < j \leq r$.
We then consider
$$
z : =  \sum_{i=1}^{n-1} {\iota_{i}}_\ast \alpha'_i - {\iota_{i , 1}}_\ast \alpha'_i \in \bigoplus_{i=1}^{n-1}  \CH_0 ( Y_{i , L} ) \oplus \CH_0 ( R_{i,1 , L} ) ,
$$
which by \eqref{eq:base-change-ramified-extension-special-fibre} is an element in the direct sum of the Chow groups of zero-cycles of the components of $\tilde Y_L$. 
It is clear from the definition of $z$ that $\deg ( z ) = 0$.
Hence, by surjectivity of $\Phi_{\widetilde{\mathcal X}_{\tilde A}}$ modulo $2$, there is a class
\[
\bar \gamma=\bar{\gamma}_{Y_1}+\sum_{j=1}^r q_{R_{1,j}}^\ast \bar{\alpha}_{R_{1,j}} + \bar{\gamma}_{Y_2} +\sum_{j=1}^r q^\ast_{R_{2,j}}\bar{\alpha}_{R_{2,j}}+\ldots + \bar{\gamma}_{Y_n} \in  \CH_1 ( \tilde{Y}_{ L} )/2  ,
\]
with $\bar{\gamma}_{Y_i}\in \CH_1 (Y_{i , L})/2$ and $\bar{\alpha}_{R_{i , j}}\in \CH_0 ( Y_{i , i +1, L})/2$, such that $\Phi_{\widetilde{\mathcal X}_{\tilde A} , R_{i , 1}}(\bar{\gamma} ) \equiv  - {\iota_{i , 1}}_\ast \alpha'_i \mod 2$ and $\Phi_{\widetilde{\mathcal X}_{\tilde A} , R_{i , j}}(\bar{\gamma} ) \equiv 0 \mod 2$ for $j>1$ an all $i$.

A similar computation as in \eqref{eq:alpha_R_j} and using Lemma \ref{lem:CH_1Y} \eqref{item:lem:CH_1Y-2} shows then that 
\begin{align*}
-\alpha'_i
&\equiv \bar{\gamma}_{Y_i}|_{Y_{i , i+ 1, L}} + \bar{\alpha}_{R_{i , 2}}\mod 2 \\
\bar{\alpha}_{R_{i ,2}}  
&\equiv \bar{\alpha}_{R_{i,4}} \equiv\ldots\equiv \bar{\alpha}_{R_{i, r-1 }}  \mod 2  \\
- \bar{\alpha}_{R_{i,r-1}}
&\equiv \bar{\gamma}_{Y_{i+1}}|_{Y_{i,i+1,L}}\mod 2 , 
\end{align*}
in $\CH_0( Y_{i, i+1, L})/2$, where we used the assumption that $r$ is odd.

Combining these equations, we obtain that $\Psi ( \sum \bar{\gamma}_{Y_i} ) = \alpha' $ modulo $2$ and so we showed that $\Psi$ from \eqref{eq:surj-mod-2} is surjective modulo $2$, as we want.

Let us now show that $\Phi_{\mathcal{X}_A}$ is surjective modulo $2$.
 Let $\beta_i\in \CH_0(Y_{i,L})/2$ for $i=1,\dots ,n$ with $\sum _i \deg(\beta_i)\equiv 0\mod 2$.
Since $\Phi_{\widetilde{\mathcal X}_{\tilde A}}$ is surjective modulo $2$, 
there is a class
$\gamma \in\CH_1(\tilde{Y}_L)$ with 
$$
\Phi_{\widetilde{\mathcal X}_{\tilde A},Y_{i,L}}(\gamma)=\beta_i\in \CH_0(Y_{i,L})/2
$$
for all $i=1,\dots ,n$
and
$$
\Phi_{\widetilde{\mathcal X}_{\tilde A},R_{i,j,L}}(\gamma)=0\in \CH_0(R_{i,j,L})/2
$$
for all $i=1,\dots ,n-1$ and $j=1,\dots ,r$.

As before, we can write 
\[
\gamma=\gamma_{Y_1}+\sum_{j=1}^r q_{R_{1,j}}^\ast\alpha_{R_{1,j}} + \gamma_{Y_2} +\sum_{j=1}^r q^\ast_{R_{2,j}}\alpha_{R_{2,j}}+\ldots + \gamma_{Y_n} ,
\]
for some $\gamma_{Y_i}\in \CH_1(Y_{i,L})$ and $\alpha_{R_{i,j}}\in\CH_{0}(Y_{i , i+1 , L})$.  
Then, as $r$ is odd, \eqref{eq:alpha_R_j} implies that
\begin{align}
\gamma_{Y_i} |_{Y_{i,i+1,L}} 
&\equiv \alpha_{R_{i,2}}\equiv\alpha_{R_{i,4}}\equiv\ldots\equiv\alpha_{R_{i, r-1 }} \equiv \gamma_{Y_{i+1}}|_{Y_{i,i+1,L}} \mod 2 , \label{eq:alpha_R_j_odd_1} \\
\alpha_{R_{i,1}}
&\equiv\alpha_{R_{i,3}}\equiv\ldots \equiv \alpha_{R_{i, r- 2 }}\equiv\alpha_{R_{{i,r}}}\mod 2 \label{eq:alpha_R_j_odd_2}
\end{align}
in $\CH_0( Y_{i, i+1, L})/2$, for $i=1,\dots ,n-1$.
Moreover, $\beta_i=\Phi_{\widetilde{\mathcal X}_{\tilde A},Y_{i,L}}(\gamma)$ modulo 2 means that
\begin{align} \label{eq:gamma-alpha-2}
\beta_1=\gamma_{Y_1}|_{Y_{1,2,L}} +\alpha_{R_{1,1}} \in \CH_0(Y_{1,L} )/2 ,
\end{align}
\begin{align} \label{eq:gamma-alpha-3}
\beta_{n} & =\gamma_{Y_n}|_{Y_{n-1,n,L}} +\alpha_{R_{n-1,r}} \notag \\
& =\gamma_{Y_n}|_{Y_{n-1,n,L}} +\alpha_{R_{n-1,1}} \in \CH_0(Y_{n,L} )/2 ,
\end{align}
and
\begin{align} \label{eq:gamma-alpha}
\beta_i & =\gamma_{Y_i}|_{Y_{i-1,i,L}} +\gamma_{Y_{i}}|_{Y_{i,i+1,L}}+\alpha_{R_{i-1,r}}+ \alpha_{R_{i,1}} \notag \\ 
& =\gamma_{Y_i}|_{Y_{i-1,i,L}} +\gamma_{Y_{i}}|_{Y_{i,i+1,L}}+\alpha_{R_{i-1,1}}+ \alpha_{R_{i,1}} \in \CH_0(Y_{i,L} )/2 ,
\end{align}
for all $2\leq i\leq n-1$.
Note that we used \eqref{eq:alpha_R_j_odd_2} in the second equation of \eqref{eq:gamma-alpha-3} and   \eqref{eq:gamma-alpha}.

By the surjectivity of $\Psi$ in (\ref{eq:surj-mod-2}), we get classes $\gamma'_{Y_i}\in \CH_1(Y_{i,L})$ with
\begin{align} \label{eq:gamma'}
\alpha_{R_{i,1}}=\gamma'_{Y_i}|_{Y_{i,i+1,L}}+\gamma'_{Y_{i+1}}|_{Y_{i,i+1,L}} \in \CH_0(Y_{i,i+1,L} )/2 ,
\end{align}
for all $i = 1 , \ldots , n-1$.
Combining this with (\ref{eq:gamma-alpha-2})--(\ref{eq:gamma-alpha}), we find
$$
\gamma_{Y_1}|_{Y_{1,2,L}} +\gamma'_{Y_1}|_{Y_{1,2,L}}+\gamma'_{Y_2}|_{Y_{1,2,L}}=\beta_1\in \CH_0(Y_{1,L} )/2 ,
$$
$$
\gamma_{Y_n}|_{Y_{n-1,n,L}} +\gamma'_{Y_{n-1}}|_{Y_{n-1,n,L}}+\gamma'_{Y_n}|_{Y_{n-1,n,L}}=\beta_n\in \CH_0(Y_{n,L} )/2 ,
$$
and
$$
\gamma_{Y_i}|_{Y_{i-1,i,L}} +\gamma_{Y_{i}}|_{Y_{i,i+1,L}}+
\gamma'_{Y_{i-1}}|_{Y_{i-1,i,L}}+\gamma'_{Y_{i}}|_{Y_{i-1,i,L}}
+ \gamma'_{Y_i}|_{Y_{i,i+1,L}}+\gamma'_{Y_{i+1}}|_{Y_{i,i+1,L}}=\beta_i \in \CH_0(Y_{i,L} )/2
$$
for $2\leq i\leq n-1$.
Let now 
$$
\gamma'=\sum_{i=1}^n \gamma'_{Y_i}\in \CH_1(Y_{L})/2\ \ \ \text{and}\ \ \ \gamma^{even}:=\sum_{j=1}^{ \lfloor n/2 \rfloor} \gamma_{Y_{2j}} \in  \CH_1(Y_{L})/2 .
$$
Using the relation $\gamma_{Y_i}|_{Y_{i,i+1,L}}=\gamma_{Y_{i+1}}|_{Y_{i,i+1,L}}$ modulo 2 from (\ref{eq:alpha_R_j_odd_1}), we then conclude from the above computation that
$$
\Phi_{\mathcal X_A,Y_{i,L}}(\gamma'+\gamma^{even})=\beta_i\in \CH_0(Y_{i,L})/2
$$
for all $i=1,\dots ,n$.
Hence, $\Phi_{\mathcal X_A}$ is surjective modulo 2, as we want.
This concludes the proof of the proposition.
\end{proof}

\subsection{Proof of Theorem \ref{thm:Phi-geom-gen}}

\begin{proof}[Proof of Theorem \ref{thm:Phi-geom-gen}] 
Since $\Phi_{\mathcal X}$ depends only on the special fibre (see Lemma \ref{lem:Phi}) and because the base change of $\mathcal{X}$ to the completion of $R$ remains a strictly semi-stable family,  we may replace $R$ by its completion and assume that $R$ is complete.
As the geometric generic fibre $X_{\bar{K}}$ of $\mathcal{X}\to\Spec R$ has a decomposition of the diagonal, it follows that there is a finite field extension $F\supset K$, such that $X_F$ has a decomposition of the diagonal.  

Let now $R_F$ be the integral closure of $R$ in $F$. 
Since $R$ is complete, $R_F$ is again a discrete valuation ring.
We consider the strictly semi-stable model $\tilde{\mathcal{X}}\to\Spec R_F$ constructed as in \eqref{eq:base-change-ramified-extension}, with special fibre $\tilde{Y}$ as described in \eqref{eq:base-change-ramified-extension-special-fibre} and with the induced morphism $\tilde{Y} \to Y$. 

The residue field $k$ of $R$ is algebraically closed by assumptions.
It follows that the residue field of $R_F$ is given by $k$ as well.
Let $A/R$ be any unramified extension of discrete valuation rings with induced extension $L/k$ of residue fields. 
We denote by $\tilde {\mathcal X}_{\tilde A}$ the base change of $\tilde {\mathcal X}$ to an unramified extension $\tilde A$ of $R_F$ that induces $L/k$ on residue fields (see Lemma \ref{lem:inflation-of-local-rings}).
By Proposition \ref{prop:Phi}, $\Phi_{\tilde {\mathcal X}_{\tilde A}}$ is surjective, hence surjective modulo 2.
It then follows from Proposition \ref{prop:Phi-and-base-change} that 
$\Phi_{\mathcal X_A}$ is surjective modulo 2 as well, as we want.
This concludes the proof of the theorem.
\end{proof}

\subsection{Proof of Theorem \ref{thm:method} and Corollary \ref{cor:method}}

\begin{proof}[Proof of Theorem \ref{thm:method}]
As aforementioned,  item (\ref{item:method:1}) of Theorem \ref{thm:method} follows from Proposition \ref{prop:Phi} and item (\ref{item:method:2}) follows from Theorem \ref{thm:Phi-geom-gen}.
\end{proof}

\begin{proof}[Proof of Corollary \ref{cor:method}]
 The main strategy here is to argue by contradiction. Indeed, we are going to assume that the geometric generic fibre has a decomposition of the diagonal and we want to conclude that $Y_{12}$ has odd torsion order, which contradicts the assumptions of the corollary. 

 Note first that the Chow group of $Y_{12}$ does not appear in the target of $\Phi_{\mathcal{X}}$ and so we need to modify our strictly semi-stable family. 
 More concretely, we  perform a $2:1$ base change and blow-up $Y_{12}$ to arrive at a strictly semi-stable model $\mathcal X'\to \Spec R'$ with special fibre $Y'=Y_1\cup P\cup Y_2$, where $q:P\to Y_{12}$ is a $\CP^1$-bundle that meets $Y_1$ and $Y_2$ along disjoint sections.
This implies $\CH_1(Y'_L)\cong \CH_1(Y_L)\oplus q^\ast \CH_0((Y_{12})_L)$, for any field extension $L/k$.
By construction of the above map, $ \Phi_{\mathcal X',P}$ is zero modulo two on $ q^\ast \CH_0((Y_{12})_L)$.
Since $\CH_1(Y)\to \CH_1(Y_L)$ is surjective by assumption, we conclude that for any unramified extension $A/R'$ of dvr's, 
\begin{align} \label{eq:im(Phi)-mod2}
 \im(\Phi_{\mathcal X'_{A},P_L})\equiv \im( \Phi_{\mathcal X',P}) \mod 2.
\end{align}
If the geometric generic fibre of $\pi$ admits a decomposition of the diagonal, then, by Theorem \ref{thm:method},  
$\Phi_{\mathcal X'_A,P_L}$ is  surjective modulo $2$ for $A$ the local ring of $\mathcal X'$ at the generic point of $P$.
By (\ref{eq:im(Phi)-mod2}), this implies that up to some multiples of $2$, $\delta_P$ is contained in $\im( \CH_0(P)\to \CH_0(P_{k(P)})$ and so $P$ has odd torsion order (see Lemma \ref{lem:torsion}).
Since the torsion order is a stable birational invariant  of smooth projective varieties, we find that $Y_{12}$ has odd torsion order as well, which contradicts our assumptions.
\end{proof}
\begin{remark}
The above proof shows that Corollary \ref{cor:method} remains true if for any field extension $L/k$, $\CH_1(Y)\to \CH_1(Y_L)$ is  surjective modulo $2$.
\end{remark}

\section{Very general quartic fivefolds have no decomposition of the diagonal}\label{sect:quartic-fivefolds-decomp-diag}

\subsection{Overview}
We aim to write down a smooth quartic fivefold $X\subset \CP^6$, which has no decomposition of the diagonal. 
We sketch our construction here, before we give the technical details below.  

As in \cite[Theorem 5.1]{NO}, we start by degenerating a general quartic fivefold to a union $Y_1 \cup Y_2$ of general quartic double covers of $\CP^5$ (see Subsection \ref{subsec:family}). 
The two components $Y_1$ and $Y_2$ are isomorphic to each other via an exchange of variables and they intersect each other in a general quartic double cover $Y_1\cap Y_2 = Z$ of $\CP^4$,  which degenerates to the double quartic fourfold $Z_0\to \CP^4$ studied by Hassett--Pirutka--Tschinkel \cite{
HPT2}. 
The total space $\mathcal X$ of this degeneration is singular along a codimension 1 subvariety $S$ of $Z$.

The blow-up $ \mathcal X':=Bl_{Y_2}\mathcal X$ resolves the singularities of the total space and its special fibre  is the union $Y_1 \cup \tilde{Y}_2$, where $\tilde{Y}_2$ is the blow-up of $Y_2$ along $S$. Moreover, this family is strictly semi-stable.
A further $2:1$ base-change, followed by a blow-up along $Z\simeq Y_1 \cap \tilde{Y}_2$, gives a strictly semi-stable family $\tilde{ \mathcal X}$ whose generic fibre is a general quartic fivefold and whose special fibre is a union $\tilde X_0=Y_1 \cup P_Z \cup \tilde{Y}_2$, where $P_Z$ is a $\CP^1$-bundle over $Z$ (see Lemma \ref{lem:main-family}).

Let $A$ be the local ring of $\tilde{ \mathcal X}$ at the generic point of $P_Z$ and with residue field $\kappa(P_Z)$, where we recall our convention  from Section \ref{subsec:convention} that the function field of an integral scheme  $X$ over a field  is denoted by $\kappa(X)$ whenever we prefer to make the ground field in our notation not explicit. 
Then $\tilde{ \mathcal X}_A\to \Spec A$ is strictly semi-stable and we get a map
$$
\Phi_{\tilde{\mathcal X}_A,P_Z}:\CH_1(\tilde X_0\times \kappa(P_Z))\longrightarrow \CH_0(P_Z\times \kappa(P_Z)) ,
$$ 
see Definition \ref{def:Phi}.
The main technical result of this section is then the assertion that for a zero-cycle $z\in \CH_0(P_Z)$ of degree one,  the zero-cycle
\begin{align}\label{eq:class-CH_0}
\delta_{P_Z}-z_{\kappa(P_Z)}\in \CH_0(P_Z\times \kappa(P_Z))
\end{align}
is not in the image of $\Phi_{\tilde{\mathcal X}_A}$ modulo 2, see Proposition \ref{prop:Obstruction-diagonal}.
By Theorem \ref{thm:Phi-geom-gen}, this implies that the geometric generic fibre of $\tilde{ \mathcal X}\to \Spec R$ has no decomposition of the diagonal, as we want.

We assume that (\ref{eq:class-CH_0}) is contained in the image of $\Phi_{\tilde{\mathcal X}_A}$ modulo 2, and aim to find a contradiction.
The strategy  is to repeatedly apply Fulton's specialization map on Chow groups (in the form of Lemma \ref{lem:specialization-functoriality} below) to simplify the contribution from $\CH_1(\tilde X_0\times \kappa(P_Z))$.
The goal will be to finally arrive at the conclusion that the diagonal point $\delta_{Z_0}\in \CH_0(Z_0\times \kappa(Z_0))$ satisfies
$$
\delta_{Z_0}\in \im(  \CH_0(Z_0 )\to  \CH_0(Z_0\times \kappa(Z_0))) \mod 2 .
$$ 
We will show (see Lemma \ref{lem:HPT-class})  that this contradicts some  properties of the non-trivial unramified cohomology class with $\Z/2$-coefficients on $Z_0$ from  \cite{HPT,HPT2}.

\subsection{A strictly semi-stable family}\label{subsec:family} 
Let $k_0$ be an algebraically closed field of characteristic different from $2$ and let 
$$
k:=\overline{k_0(\lambda, u,v,s)}
$$
be the algebraic closure of purely transcendental extension of $k_0$ of degree four.

Our construction relies on the following choices.

\begin{definition}\label{def:fs-gu}
Let $f,g \in k_0[z_0 , \ldots , z_4]$ be general homogeneous polynomials with $\deg(f)=4$ and $\deg(g)=3$. 
\begin{enumerate}
\item 
Let 
$$
f_s:=sf+f_0 \in  k [z_0 , \ldots , z_4 ]
$$
where
\begin{equation}\label{eq:HPT}
f_0 := z_0 z_1 z_3^2 + z_0 z_2 z_4^2 + z_1 z_2 (z_0^2 + z_1^2 + z_2^2 - 2( z_0 z_1 + z_0 z_2 + z_1 z_2)) .
\end{equation} 
\item Let \begin{align}\label{def:gu}
g_u:=ug+z_2^3\in k[z_0 , \ldots , z_4].
\end{align} 
\item For $f_s$ and $g_u$ as above, we define
\begin{align}\label{def:F}
F :=F_{v,u,s}:= v (z_5^4 + z_6^4) + (z_5 + z_6 ) g_u + f_s \in k[z_0 , \ldots , z_6 ] ,
\end{align}
which is symmetric in $z_5 $ and $z_6$.
\end{enumerate}
\end{definition}

\begin{lemma}\label{lem:barX_1}
Let $f_s$ and $g_u$ be as in Definition \ref{def:fs-gu}.
Then
$$
\{f_s=0\}\subset \CP^4_k, \ \ \{g_u=0\}\subset \CP^4_k,\ \ \ \text{and}\ \ \{ f_s =  g_u = 0 \} \subset \CP^4_{k}
$$
are smooth complete intersections.
Moreover,
\begin{equation}\label{eq:equation-of-Y_1}
D_1 := \{ v z_5^4 + z_5 g_u + f_s = 0 \} \subset \CP^5_{ k}
\end{equation}
is smooth and
\begin{equation}\label{eq:equation-of-singular-X_1}
\bar{D}_1 := \{  z_5 g_u + f_s = 0 \} \subset \CP^5_{ k }
\end{equation}
has a singularity of multiplicity three at $p= [0 : \ldots : 0 : 1]$  and is smooth away from $p$.
\end{lemma}
\begin{proof}
It suffices to prove the assertion after specializing $s\to \infty$ and $u\to \infty$, i.e.\ we may replace $f_s$ by $f$ and $g_u$ by $g$.
Since $f$ and $g$ are general,  we thus conclude from Bertini's theorem (which works over infinite fields) that 
 $\{f_s=0\}\subset \CP^4_k$ and  $\{g_u=0\}\subset \CP^4_k$ are smooth hypersurfaces that meet in a smooth complete intersection $\{ f_s =  g_u = 0 \} \subset \CP^4_{k}$.
 This proves the first assertion of the lemma; the rest follows easily from this. 
\end{proof}

Let $R:=k[[t]]$ and let $\CP_R(1^7 , 2): = \Proj (R[z_0 ,\ldots , z_6 , w])$ be the weighted projective space, such that the variables $z_i$ have weight one and the variable $w$ has weight two.  
Consider the $R$-scheme
\begin{equation}\label{eq:main-family}
\mathcal X:=\{( \lambda w - z_2^2 )t - z_5 z_6=0 ,\, w^2 - F =0\}\subset \CP_R(1^7 , 2),
\end{equation}
where $F$ is as in (\ref{def:F}).
The $R$-scheme $\mathcal{X}$ is flat over $R$ and the generic fibre $X_K$ is given by the equation
\[
X_K = \{ \lambda^{-2} ( t^{-1}z_5 z_6 + z_2^2 )^2 - v( z_5^4 + z_6^4 ) - ( z_5 + z_6 ) g_u - f_s = 0 \} \subset \CP^6_K,  
\]
which is a smooth quartic fivefold.
(We note that $\mathcal X$ is not regular; in fact, the Weil divisors $Y_1$ and $Y_2$ are not Cartier.)

The special fibre $X_0$ of $\mathcal{X}$ has two irreducible components $X_0=Y_1\cup Y_2$, such that 
\begin{equation}\label{eq:double-quartic-fivefold}
Y_i =\{ w^2 - F =0, z_{7 - i}=0\} = \{ w^2 - v z_{i+4}^4 - z_{i+4} g_u - f_s = 0 \}  \subset \CP_{k} (1^6 , 2)
\end{equation}
for $i=1,2$ is a smooth double quartic fivefold $Y_i\to \CP ^5_{[z_0 : \ldots : z_4 : z_{4+i}]}$. 
Using the symmetry of $F$ in $z_5$ and $z_6$, we get a canonical isomorphism $Y_1\simeq Y_2$.
Moreover, the intersection $Z:=Y_1\cap Y_2 \subset \CP(1^5 ,2)$ is the double quartic fourfold $Z\to \CP^4 $ given by the equation
\begin{equation}\label{eq:smooth-fourfold-double-cover}
Z := \{ w^2 - f_s = 0 \} \subset \CP_k (1^5 ,2 ) . 
\end{equation}
Since $\{f_s=0\}$ is smooth, so is $Z$.

\begin{lemma} \label{lem:mathcalX-alongS}
The singular locus of  $\mathcal{X}$ in \eqref{eq:main-family} is given by
\begin{equation}\label{eq:B_lambda}
S = \{ t = z_5 = z_6 = \lambda w - z_2^2 = w^2 - f_s = 0 \} \subset \mathcal{X} ,
\end{equation}
which is smooth over $k$.
Moreover,  $\mathcal X$ has ordinary quadratic singularities of codimension 3 along $S$.
\end{lemma}
\begin{proof}
As described above, the special fibre $X_0$ is given by the union $Y_1\cup Y_2$ of two smooth projective varieties.
It follows that $\mathcal X$ is regular outside of $Y_1\cap Y_2$ and so the singular locus $S$ is contained in $\{t=z_5=z_6=0\}$.
Considering the Jacobian matrix of $\mathcal{X}$,  the Jacobian criterion then easily shows that $S$ is given as claimed in (\ref{eq:B_lambda}).
Since $\lambda\in k$ is nonzero, $S$ is isomorphic to
$$
\{\lambda^{-2}z_2^4 - f_s = 0 \}\subset \CP^4_k ,
$$
which is smooth over $k$ because $f_s=sf+f_0$ with $f$ general.

Let $p\in S$.
Since $S\subset \{z_5=z_6=0\}$, the $z_5$- and $z_6$-coordinate of $p$ vanish and so at least one of the $z_i$-coordinates with $i=0,\dots ,4$ is nonzero.
Since $\{f_s=0\}\subset \CP^4$ is smooth, we further see that for $i=0,\dots ,4$, at least one partial $\del_{z_i}F=\del_{z_i} f_s$ does not vanish at $p$ (if they vanished all simultaneously, then $f_s$ would vanish at $p$ by the Leibniz identity, contradicting the smoothness of $\{f_s=0\}\subset \CP^4$).
It follows that the tangent space of $\{w^2-F=0\}$ at $p$ intersects the tangent cone of $\{ ( \lambda w - z_2^2 )t - z_5 z_6=0\}$ at $p$ transversely.
The latter is Zariski locally isomorphic to the tangent cone of the ordinary quadratic singularity $\{xt-yz=0\}$, thus proving the claim in the lemma.
\end{proof}

\begin{lemma} \label{lem:main-family-0}
Let $\mathcal{X}$ be the $R= k [[t]]$-scheme as defined in \eqref{eq:main-family}, let $K := k (( t ))$, and let $Y_1$, $Y_2$  be the components of the special fibre $X_0$ of $\mathcal{X}$.  
Then $\mathcal{X}' : = Bl_{Y_2} \mathcal{X}$ is strictly semi-stable with special fibre $Y_1\cup \tilde Y_2$, where $\tilde Y_2=Bl_SY_2$ and $Y_1\cap \tilde Y_2=Bl_SZ=Z$, where $Z= Y_1 \cap Y_2$.
\end{lemma}

\begin{proof}
Note that $Y_1\to \CP^5$ is a double cover branched along $D_1$.
By Lemma \ref{lem:barX_1}, $D_1$ is smooth and so $Y_1$ is smooth.
Since $Y_2\cong Y_1$, the same holds true for $Y_2$.
We have seen above that the singular locus $S$ of $\mathcal X$ is also smooth.
Locally at a point of $S$, $\mathcal X$ has ordinary quadratic singularities of codimension  (see Lemma \ref{lem:mathcalX-alongS}) 
and a local computation shows that the special fibre of $\mathcal X'$ is given by $Y_1\cup \tilde Y_2$, where $\tilde Y_2=Bl_SY_2$.
Since $Y_2$ and $S$ are smooth, so is $\tilde Y_2$.
Moreover, $Y_1\cap \tilde Y_2=Bl_SZ=Z$, where the second equality comes from the fact that $S\subset Z$ is a divisor and $Z$ is smooth.
By construction, $\tilde Y_2$ is a Cartier divisor of $\mathcal X$; since the components of $X'_0$ are reduced and $X_0$ is Cartier, we find that $Y_1$ is Cartier as well.
Since $Y_1$, $\tilde Y_2$ and $Y_1\cap \tilde Y_2=Z$ are smooth and the components of $X'_0$ are Cartier, it follows that $\mathcal X'$ is strictly semi-stable, as we want.
\end{proof}

\begin{lemma}\label{lem:main-family}
In the notation of Lemma \ref{lem:main-family-0}, let $\mathcal{X}'' = \mathcal{X}'\times _{\substack{R\to R\\ t\to t^2}} R$.
Then 
\begin{equation}\label{eq:double-quartic-family}
\tilde{\mathcal{X}} := Bl_{Z}\mathcal{X}'' \to \Spec R
\end{equation}
is a strictly semi-stable $R$-scheme with special fibre
\[
\tilde{X}_0 = Y_1 \cup P_Z \cup \tilde{Y}_2 ,
\]
where $\tilde{Y}_2$ is the blow-up of $Y_2$ along $S$ and $P_Z$ is a $\CP^1_{\kappa}$-bundle over $Z$.  
The intersections $Y_1\cap P_Z$ and $ P_Z\cap \tilde{Y}_2$ are disjoint sections of $P_Z\to Z$.
The generic fibre
\begin{equation}\label{eq:quartic-fivefold}
\tilde X_K = \{  \lambda^{-2} ( t^{-2}z_5 z_6 + z_2^2 )^2 - v (z_5^4 + z_6^4 ) - (z_5 + z_6) g_u - f_s = 0\} \subset \CP^6_K 
\end{equation}
of $\tilde{\mathcal{X}}$ is a smooth quartic fivefold.
\end{lemma}

\begin{proof}
By Lemma \ref{lem:main-family-0}, $\mathcal X'\to \Spec R$ is strictly semi-stable. 
The $2:1$ base change $\mathcal X''$ is thus regular away from the singular locus $Z$ of the central fibre and it has ordinary double point singularities along $Z$ (because \'etale locally at the non-smooth locus, $\mathcal X'$ is given by $t=xy$ and so $\mathcal X''$ is given by $t^2=xy$, cf.\ \cite[Proposition 1.3]{hartl}).
Those singularities are resolved by the blow-up of $Z$ and the corresponding exceptional divisor will be a reduced component of the special fibre,  see e.g. \cite[Proposition 2.2]{hartl}). 
Hence, $\tilde{\mathcal X}$ is regular and the special fibre is given by $ Y_1 \cup P_Z \cup \tilde{Y}_2 $, where $P_Z\to Z$ is a smooth conic bundle.
Moreover, $Y_1\cap P_Z$ (as well as $\tilde Y_2\cap P_Z$) is a section of $P_Z\to Z$ and so $P_Z\to Z$ is a Zariski locally trivial $\CP^1$-bundle, as claimed.
This proves the lemma.
\end{proof}
 
\subsection{Main result} \label{subsec:main-result}  

\begin{theorem} \label{thm:quartic-fivefold-body} 
Let $\overline K$ be the algebraic closure of the fraction field $K$ of $R$.
Then the smooth quartic fivefold
$
\tilde X_{\overline K}\subset \CP^6_{\overline K}
$
given by the base change of \eqref{eq:quartic-fivefold} to $\overline K$ 
does not admit a decomposition of the diagonal.
\end{theorem}
 
\begin{proof}
We aim to deduce Theorem \ref{thm:quartic-fivefold-body} from Theorem \ref{thm:Phi-geom-gen}.
To this end,  let $A=\mathcal{O}_{\tilde{\mathcal{X}}, P_Z}$ with residue field $\kappa (P_Z)$.
Then $R\to A$ is an unramified extension of dvr's and so it follows from Lemma \ref{lem:main-family} that $\tilde{ \mathcal X}_A\to \Spec A$ is strictly semi-stable.
By  Definition \ref{def:Phi}, we get a map
$$
\Phi_{\tilde{\mathcal X}_A,P_Z}:\CH_1(\tilde X_0\times \kappa(P_Z))\longrightarrow \CH_0(P_Z\times \kappa(P_Z)) .
$$  
By Theorem \ref{thm:Phi-geom-gen},  Theorem \ref{thm:quartic-fivefold-body} follows 
from Proposition \ref{prop:Obstruction-diagonal} below, which is the main technical result of this section.
\end{proof}

\begin{proposition}\label{prop:Obstruction-diagonal}
Let $\tilde{\mathcal{X}}\to \Spec R$ be as in \eqref{eq:double-quartic-family} and let $A=\mathcal{O}_{\tilde{\mathcal{X}}, P_Z}$ with residue field $\kappa (P_Z)$.  
Then for any zero-cycle $z\in \CH_0(P_Z)$, the element
\begin{align}
\label{eq:class-CH_0-main}
\delta_{P_Z}-z_{\kappa(P_Z)}\in \CH_0(P_Z\times \kappa(P_Z))/2
\end{align}
is not in the image of $\Phi_{\tilde{\mathcal X}_A,P_Z}$ modulo 2, where $\delta_{P_Z}$ denotes the diagonal point of  $P_Z\times \kappa(P_Z)$. 
\end{proposition}

\begin{proof}
Recall from Lemma \ref{lem:main-family} that the special fibre of $\tilde {\mathcal X}\to \Spec R$ is given by $Y_1 \cup P_Z \cup \tilde{Y}_2$, where $\tilde Y_2\to Y_2$ is the blow-up along the smooth subvariety $S\subset Z\subset Y_2$.
Since blow-ups commute with extensions of the base field, the blow-up formula for Chow groups yields a canonical isomorphism
$$
 \CH_1(Y_2\times \kappa(P_Z))\oplus \CH_0(S\times \kappa(P_Z)) \cong \CH_1(\tilde Y_2\times \kappa(P_Z)),
$$
Since $P_Z$ is a $\CP^1$-bundle over $Z$, 
$\CH_0(Z\times \kappa(P_Z))\oplus \CH_1(Z\times \kappa(P_Z))\cong \CH_1(P_Z\times \kappa(P_Z)) $.
Since $Y_1\cap P_Z$ is a section of $P_Z\to Z$,  the contribution of $ \CH_1(Z\times \kappa(P_Z))$  to $\CH_1(\tilde{X}_0 \times \kappa (P_Z))$ is absorbed by $\CH_1(Y_1\times \kappa(P_Z)) $ and we get a canonical surjection
$$
\CH_1(Y_1\times \kappa(P_Z))\oplus \CH_0(Z\times \kappa(P_Z))\oplus  \CH_1(Y_2\times \kappa(P_Z))\oplus \CH_0(S\times \kappa(P_Z)) \twoheadrightarrow
\CH_1(\tilde X_0\times \kappa(P_Z) ) .
$$ 
We aim to compute the image of $\Phi_{\tilde{\mathcal X}_A,P_Z}$ modulo $2$.
By Lemma \ref{lem:CH_1Y} \eqref{item:lem:CH_1Y-2}, the contribution of $\CH_0(Z \times \kappa(P_Z))$ via $ \Phi_{\tilde{\mathcal X}_A,P_Z}$ is divisible by $2$ and so we may neglect it in what follows.
 Using the symmetry $Y_1\cong Y_2$ together with the fact that $Y_1$ and $\tilde Y_2$ meet $P_Z$ in (disjoint) sections of the $\CP^1$-bundle $P_Z\to Z$,  we thus conclude that
\begin{align}\label{eq:im(Phi-X)-degeneration}
\im(\Phi_{\tilde{\mathcal X}_A,P_Z})=\im( \CH_1(Y_1\times \kappa(P_Z))\oplus  \CH_0(S \times \kappa(P_Z)) \to  \CH_0(P_Z\times \kappa(P_Z)))  \mod 2.
\end{align}
Here $ \CH_1(Y_1\times \kappa(P_Z))\to  \CH_0(P_Z\times \kappa(P_Z))$ is given by restriction to the Cartier divisor $(Y_1\cap P_Z)\times \kappa(P_Z)$ and pushing forward the resulting zero-cycle to $P_Z\times \kappa(P_Z)$.
Moreover, the map $\CH_0(S\times \kappa(P_Z))\to  \CH_0(P_Z\times \kappa(P_Z))$ is the natural pushforward map, where we use the section of the $\CP^1$-bundle $P_Z\to Z$ given by $P_Z\cap \tilde Y_2$ to identify $S \subset Z$ with a subscheme of $P_Z$.

 The idea is now to perform some specializations to $Y_1$,  $S$, and $P_Z$, to make their Chow groups more accessible, so that we can control the image in (\ref{eq:im(Phi-X)-degeneration}).
 The main technical tool which allows us to perform these specializations is the following result of Fulton.

 \begin{lemma}\label{lem:specialization-functoriality}
Let $B$ be a discrete valuation ring with fraction field $F$ and residue field $L$.
Let $p:\mathcal X\to \Spec B$ and $q:\mathcal Y\to \Spec B$ be a flat proper $B$-schemes with connected fibres.
Denote by $X_\eta,Y_\eta$ and $X_0,Y_0$ the generic and special fibres of $p$, $q$, respectively.
Assume that there is a component $Y_0'\subset Y_0$, such that 
 $A=\mathcal O_{\mathcal Y,Y'_0}$ is a discrete valuation ring (this holds if $Y_0$ is reduced along $Y_0'$) and consider the flat proper $A$-scheme $\mathcal X_A\to \Spec A$, given by base change of $\pi$.
Then Fulton's specialization map induces a specialization map
$$
\spe:\CH_i( X_\eta \times_{F} \overline{F}(Y_\eta) )\longrightarrow \CH_i( X_0 \times_{L} \overline{L}(Y'_0) ),
$$
where $\overline F$ and $\overline L$ denote the algebraic closures of $F$ and $L$, respectively, such that the following holds:
\begin{enumerate}
\item $\spe$  commutes with pushforwards along proper maps and pullbacks along regular embeddings;
\item If $\mathcal X=\mathcal Y$ and $X_0$ is integral, then $\spe(\delta_{X_\eta})=\delta_{X_0}$, where $\delta_{X_\eta}\in \CH_0( X_\eta \times_{F} \overline{F}(X_\eta))$ and $\delta_{X_0}\in \CH_0( X_0 \times_{L} \overline{L}(X_0))$ denote the diagonal points.
\end{enumerate} 
\end{lemma}

\begin{proof}
Let $n$ be the relative dimension of $q:\mathcal Y\to \Spec B$.
Consider the flat proper $B$-scheme $p\times q:\mathcal X\times_B \mathcal Y\to \Spec B$.
By \cite[\S 20.3]{fulton} (see also  \cite[Theorem 3.3 (b)]{fulton2}), there is a specialization map
\begin{align} \label{eq:spe}
\spe:\CH_{i+n}( X_\eta \times_{F} Y_\eta \times_F \overline F  )\longrightarrow \CH_{i+n}( X_0 \times_{L} Y_0 \times_L \overline{L}  ).
\end{align}
This map is compatible with respect to pushforwards along proper maps and pullbacks along regular embeddings, see \cite[Proposition 20.3]{fulton}.
For a given cycle $\gamma$ on $X_\eta \times_{F} Y_\eta \times_F \overline F $, the specialization $\spe(\gamma)$ is constructed by first performing a base change so that $\gamma$ is defined on the generic fibre of $p\times q$.
We may then take the closure $\overline \gamma$ of $\gamma$ in the total space $\mathcal X\times_B\mathcal Y$ and restrict $\overline \gamma$ to the special fibre.
The cycle $\spe(\gamma)$ is then given by the image of $\overline \gamma|_{X_0\times_L Y_0}$ via the natural map
$$
\CH_{i+n}( X_0 \times_{L} Y_0   )\to \CH_{i+n}( X_0 \times_{L} Y_0 \times_L \overline{L}  ).
$$
(Taking the image via the latter map is necessary to make the construction well-defined, because the base change performed above may replace $L$ by a finite extension.)

Pullback of cycles yields a canonical isomorphism
$$
\lim_{\substack{\longrightarrow \\ \emptyset\neq U\subset Y'_0}} \CH_{i+n}( X_0 \times_{L} U \times_L \overline{L}  ) \stackrel{\cong}\longrightarrow \CH_i( X_0 \times_{L} \overline{L}(Y'_0) ),
$$
whose inverse is given by taking closures of cycles.
We may thus consider the induced 
pullback (resp.\ localization) map
\begin{align}\label{eq:localization}
\CH_{i+n}( X_0 \times_{L} Y_0 \times_L \overline{L}  )\longrightarrow  \CH_i( X_0 \times_{L} \overline{L}(Y'_0) ).
\end{align}
By the localization exact sequence \cite[Proposition 1.8]{fulton},  the kernel of the above map is generated by cycles on $ X_0 \times_{L} Y_0 \times_L \overline{L} $ that do not dominate $Y'_0 \times_L \overline{L}$.  
Using this together
with the above description of the specialization map in (\ref{eq:spe}), we find that the composition of (\ref{eq:spe}) with (\ref{eq:localization}) 
factorizes through  
$$
\CH_{i+n}( X_\eta \times_{F} Y_\eta \times_F \overline F  )\longrightarrow \CH_i( X_\eta \times_{F} \overline{F}(Y_\eta) ) ,
$$ 
as the kernel of the latter is again generated by cycles that do not dominate $Y_\eta\times_F \overline F$.
We thus arrive at the specialization map as claimed in the lemma.
The compatibility results with respect to proper pushforwards and pullbacks along regular embeddings follow from the respective properties for (\ref{eq:spe}) and the claim concerning the image of the diagonal point is clear from the construction.
\end{proof}

Since the specialization map in Lemma \ref{lem:specialization-functoriality} commutes with pushforwards along proper maps and pullbacks along regular embeddings,  we may compute the specialization of   (\ref{eq:im(Phi-X)-degeneration})  simply by specializing the involved varieties.
The assumption that (\ref{eq:class-CH_0}) lies in  (\ref{eq:im(Phi-X)-degeneration}) modulo $2$ implies then that the same holds true after specialization and we aim to finally arrive at a contradiction after specializing the parameters $\lambda$, $v$, $u$ and $s$ to zero.

Note that $Z$ depends on $s$, but not on $\lambda$, $v$,  and $u$, and so we write from now on $Z=Z_s$.
Similarly, $S$ depends on $s$ and $\lambda$ but not on $u$ and $v$, and so we write $S=S_{s,\lambda}$.
Note also that $Y_1$ depends on $v$, $u$, and $s$ but not on $\lambda$, but we will not need to indicate this in our notation.
The required degenerations are captured in the following diagram:
\begin{align*}
S_{s,\lambda} \subset Z_s \subset P_{Z_s} {\overset{\lambda \to 0}{\longsquiggly}} &  S_s \subset Z_s  \subset P_{Z_s} \\
 Y_{1 } \xrightarrow{2:1} \CP^5 {\overset{v \to 0}{\longsquiggly}} & \bar{Y}_1 \xrightarrow{2:1} \CP^5 \\
 T_{s,u} \subset Z_s {\overset{u \to 0}{\longsquiggly}} & T_s \subset Z_s \\
S_s\subset Z_s,\ \ \ T_s \subset Z_s  {\overset{s \to 0}{\longsquiggly}} & S_0^{\red}=T_0^{\red} \subset Z_0  .
\end{align*}
Here $T_{s,u}\subset Z_s\subset P_{Z_s}$ is a hypersurface of $Z_s$ that depends on $s$ and $u$ and such that 
$$
\im(\CH_1(\bar Y_1 \times \kappa(P_{Z_s}) )\to \CH_0(P_{Z_s}  \times \kappa(P_{Z_s}) ) )\subset \im(\CH_0(T_{u,s} \times \kappa(P_{Z_s}))\to \CH_0(P_{Z_s} \times \kappa(P_{Z_s})  )  ).
$$

We explain our notation and construction in detail in the following four steps.
In each step Lemma \ref{lem:specialization-functoriality} will be applied to a situation where the special fibre $Y_0$ of $\mathcal Y$ is integral, so that $Y_0'=Y_0$.
To simplify notation, we will not write down the total spaces of our degenerations explicitly, but only indicate which parameter (i.e.\ $\lambda$, $v$, $u$, or $s$) is sent to zero.
For the same reason, we will use the following convention: if the defining equation of $X$ does not depend on a parameter $\mu$, then we denote  the specialization of $X$ via $\mu\to 0$ with the same letter.
When applying the above specializations, the ground field $k=\overline{k_0(\lambda,u,v,s)}$ will change in each step in the sense that one has to delete one transcendental parameter each time (note however that the resulting field remains algebraically closed by construction in Lemma \ref{lem:specialization-functoriality}). 
To simplify notation, we will  not make this change of the ground field explicit in our notation.

\subsubsection{Step 1}
In the first step, we aim to simplify the contribution from $S=S_{s,\lambda}$ by specializing $\lambda \to 0$.
The union $Y_1\cup P_{Z_s}$ does not depend on $\lambda$.
On the other hand,  $S_{s,\lambda}$ specializes by (\ref{eq:B_lambda}) to the hypersurface
\begin{align} \label{eq:S_s}
S_s= \{ z_5 = z_6 = z_2^2 = w^2 - f_s = 0 \}
\end{align}
given as the pullback of the non-reduced plane $z_2^2=0$ via the double covering $Z_s\to \CP^4$.
We thus find by Lemma \ref{lem:specialization-functoriality} that the image of (\ref{eq:im(Phi-X)-degeneration}) via the specialization $\lambda\to 0$ is given by
\begin{align} \label{eq:im(S_s)}
 \im( \CH_1(Y_1\times \kappa(P_{Z_s}))\oplus  \CH_0(S_s\times \kappa(P_{Z_s})) \to  \CH_0(P_{Z_s}\times \kappa(P_{Z_s})))  \mod 2 .
\end{align}

\subsubsection{Step 2}
In the second step, we aim to get a hand on $\CH_1(Y_1)$.
For this, we degenerate $Y_1$ via $v\to 0$, so that the double cover $Y_1\to \CP^5$ from (\ref{eq:double-quartic-fivefold}) specializes to a singular double cover $\bar{Y}_1\to \CP^5$, branched along the quartic 
$ 
\bar D_1\subset \CP^5
$ from (\ref{eq:equation-of-singular-X_1}).
By Lemma \ref{lem:barX_1},  $\bar D_1$ has a triple point at $p=[0:\dots :0:1]$ as its unique singularity. 
The triple point makes $\bar{Y}_1$ rational and we will use this to control its first Chow group.
Since $Z_s$, as well as $S_s$, do not depend on the parameter $v$,  they specialize smoothly in this step.
Hence the image of (\ref{eq:im(Phi)-mod2}) via the specialization $\lambda\to 0$ followed by $v\to 0$ is given by
\begin{align} \label{eq:im(overlineY)}
 \im( \CH_1(\overline Y_1\times \kappa(P_{Z_s}))\oplus  \CH_0(S_s\times \kappa(P_{Z_s})) \to  \CH_0(P_{Z_s}\times \kappa(P_{Z_s})))  \mod 2 .
\end{align}

Let $\hat Y_1$ be the blow-up of  $\overline Y_1$ along $p$.
Projection from $p$ given by $[z_0:\dots:z_5]\mapsto [z_0:\dots:z_4]$ induces a morphism
 $$
 f:\hat Y_1\longrightarrow \CP^4.
 $$
Let 
$$ 
C_{u}:=\{g_u=0\}\subset \CP^4.
$$

\begin{lemma}\label{lem:hatY1} 
The base change of $f:\hat Y_1\to \CP^4$  to the open subset $U=\CP^4\setminus C_u$ is a Zariski-locally trivial $\CP^1$-bundle.
\end{lemma}
\begin{proof}
We identify $\CP^4$ with the hyperplane $H=\{z_5=0\}\subset \CP^5$.
For a point $y\in H$ that is not contained in $C_u$, the line through $p$ and $y$ in $\CP^5$ meets $\bar D_1$ in exactly two points: in $p$ with multiplicity three and in a another point $q(y)$ with multiplicity one.
The fibre $f^{-1}(y)$ then identifies to the double cover of $\CP^1$ branched at $p$ and $q(y)$.
Hence, $f^{-1}(y)$ is a smooth conic and the branch point $p$, which does not depend on $y$ yields a section of the base change of $f$ to $U$.
Hence, $f^{-1}(U)\to U$ is a smooth conic bundle with a section,  and so it identifies to a Zariski locally trivial $\CP^1$-bundle.
This proves the lemma.
\end{proof}

\begin{remark}
In the above notation, one can check that for $y\in C_u$, the point $q(y)$ collides with $p$ and the fibre $f^{-1}(y)$ is given by two lines that meet in one point, corresponding to the point $p$.
We did not include this description as it will be irrelevant for our argument.
\end{remark}

\begin{corollary}\label{cor:CH1(hatY1)}
The  canonical map $\CH_1(f^{-1}(C_u))\to \CH_1(\hat Y_1)$ is universally surjective, i.e.\ it is surjective after any extension of the base field.
\end{corollary}
\begin{proof}
Since $f^{-1}(U)\to U$ is a $\CP^1$-bundle, by Lemma \ref{lem:hatY1},  $\CH_1(f^{-1}(U))\cong \CH_0(U) \oplus \CH_1 (U)$ holds universally, i.e.\ after any extension of the base field.
By the localization exact sequence, we get an exact sequence $\CH_0(C_u)\to \CH_0(\CP^4)\to \CH_0(U)\to 0$, which holds again after any extension of the base field.
Since $C_u$ contains a rational point (as it is defined over an algebraically closed field), the first arrow in the localization exact sequence is surjective and so $\CH_0(U)=0$ holds after any extension of the base field.
Similarly, $\CH_1(C_u)\to \CH_1(\CP^4)\to \CH_1(U)\to 0$ is exact after any extension of the base field.
Since $C_u$ is a cubic threefold over an algebraically closed field, it contains a line and so we conclude that $\CH_1(U)=0$ holds after any extension of the base field.
Altogether we find that $\CH_1(f^{-1}(U))=0$ after any extension of the base field.
The result follows thus from the localization exact sequence $\CH_1(f^{-1}(C_u))\to \CH_1(\hat Y_1)\to \CH_1(f^{-1}(U))\to 0$.
This proves the corollary. 
\end{proof}

Consider \eqref{eq:im(overlineY)} and recall that the map
\begin{align}\label{eq:im(barY-1-2)}
 \CH_1(\overline Y_1\times \kappa(P_{Z_s}))\longrightarrow \CH_0(P_{Z_s}\times \kappa(P_{Z_s}))
\end{align}
is given by intersecting a one-cycle on $\overline Y_1\times \kappa(P_{Z_s})$ with 
$$
(\overline Y_1\cap P_{Z_s})\times  \kappa(P_{Z_s})\cong Z_s\times  \kappa(P_{Z_s}) .
$$
The singular point of $\overline Y_1$ does not meet the above intersection.
To compute the image of \eqref{eq:im(barY-1-2)}, we may thus replace $\overline Y_1$ by the blow-up $\hat Y_1$.
Corollary \ref{cor:CH1(hatY1)} then shows that the image of \eqref{eq:im(barY-1-2)}  is contained in the image of
 \begin{align*} 
 \CH_0(T_{s,u}\times \kappa(P_{Z_s}))\longrightarrow \CH_0(P_{Z_s}\times \kappa(P_{Z_s})),
\end{align*}
where
\begin{align} \label{eq:T_su}
T_{s,u}:=f^{-1}(C_u)\cap Z_s=\{w^2-f_s=g_u=0\}\subset \CP(1^5,2),
\end{align}
cf.\ \eqref{eq:smooth-fourfold-double-cover}.
We thus conclude that  \eqref{eq:im(overlineY)} is contained in
\begin{align} \label{eq:im(T_u)}
 \im( \CH_0(T_{s,u} \times \kappa(P_{Z_s}))\oplus  \CH_0(S_s\times \kappa(P_{Z_s})) \to  \CH_0(P_{Z_s}\times \kappa(P_{Z_s})))  \mod 2 .
\end{align}

\subsubsection{Step 3}
In this step we specialize $u\to 0$.
By construction of $g_u$ in \eqref{def:gu}, this specializes $g_u$ to $g_0=z_2^3$ and hence $T_{s,u}$ specializes to
$$
T_s:=\{w^2-f_s=z_2^3=0\}\subset \CP(1^5,2).
$$ 
Applying the specialization map from Lemma \ref{lem:specialization-functoriality} that corresponds to $u\to 0$ to (\ref{eq:im(T_u)}), we then get
\begin{align*} 
 \im( \CH_0(T_{s } \times \kappa(P_{Z_s}))\oplus \CH_0(S_s\times \kappa(P_{Z_s})) \to  \CH_0(P_{Z_s}\times \kappa(P_{Z_s})))  \mod 2 .
\end{align*}
Comparing the above description of $T_s$ with that of $S_s$ in \eqref{eq:S_s}, we find that $T_s^{\red}=S_s^{\red}$.
Since Chow groups depend only on the underlying reduced schemes, the above image simplifies further to 
 \begin{align} \label{eq:im(T_s)-2}
 \im(  \CH_0(T^{\red}_s\times \kappa(P_{Z_s}) ) \to  \CH_0(P_{Z_s}\times \kappa(P_{Z_s})))  \mod 2 .
\end{align}

\subsubsection{Step 4}
In this final step, we specialize $s\to 0$.
This way $Z_s$ specializes to the double cover $Z_0\to \CP^4$ branched along the quartic $f_0$ of Hassett--Pirutka--Tschinkel from \eqref{eq:HPT}.
Moreover, $T^{\red}_s$ specializes to
$$
T^{\red}_0=\{w^2-f_0=z_2=0\}\subset \CP(1^5,2).
$$
Using the explicit description of $f_0$ in (\ref{eq:HPT}), we find
$$
T^{\red}_0=\{w^2-z_0z_1z_3^2=z_2=0\}\subset \CP(1^5,2).
$$

\begin{lemma}\label{lem:W_0-rational}
$T_0^{\red}$
has universally trivial Chow group of zero-cycles.
\end{lemma}

\begin{proof}
We write for simplicity $T:=T_0^{\red}$.
The non-normal locus of $T $ is given by the plane $\{w=z_3=0\}$.
The normalization $\tilde{T}\to T$ of $T$ is locally given by the equation
\[
  v^2 - z_0 z_1 =0 ,
\]
where we substituted $v=w/ z_3$. 
That is, $\tilde T$ is an integral quadric given by the equation
\[
\tilde T=\{ z_4^2 - z_0 z_1 = 0\} \subset \CP^4_{[z_0 : \ldots : z_4]}.
\]
Note that $\tilde T$ is a cone over a smooth conic with a rational point, hence a cone over $\CP^1 $ and so it has universally trivial Chow group of zero-cycles.
Since the non-normal locus of $T$ is a plane, which of course has universally trivial Chow group of zero cycles as well, we conclude that $T$ has universally trivial Chow group of zero-cycles, as we want.  
\end{proof}

By Lemma \ref{lem:W_0-rational}, the specialization of (\ref{eq:im(T_s)-2}) under $s\to 0$ is contained in 
\begin{align} \label{eq:im(T_0)}
 \im( \CH_0(P_{Z_0})  \to  \CH_0(P_{Z_0}\times \kappa(P_{Z_0})))  \mod 2 .
\end{align}
Our initial assumption that (\ref{eq:class-CH_0-main}) is contained in (\ref{eq:im(Phi-X)-degeneration}) then implies  that
$ 
\delta_{P_{Z_0}} 
$ 
is contained in (\ref{eq:im(T_0)}).
Since $P_{Z_0}$ is a $\CP^1$-bundle over $Z_0$,  the pushforward of cycles yields a universal isomorphism $\CH_0(P_{Z_0})\cong \CH_0( Z_0 )$.
Moreover,  $\CH_0( Z_0 \times \kappa( Z_0))\cong \CH_0( Z_0 \times \kappa(P_{Z_0}))$, because Chow groups do not change under purely transcendental field extensions.
Since $\delta_{P_{Z_0}} $ maps to $ \delta_{Z_0 }$ via the composition
$$
 \CH_0(P_{Z_0}\times \kappa(P_{Z_0}))\stackrel{\cong}\longrightarrow  \CH_0( Z_0 \times \kappa(P_{Z_0})) \stackrel{\cong}\longrightarrow  \CH_0( Z_0 \times \kappa( Z_0))
$$ 
we thus conclude 
  that  
$$
 \delta_{Z_0 }\in  \im( \CH_0( Z_0 )  \to  \CH_0( Z_0 \times \kappa(Z_0)))  \mod 2.
$$
The proof of Proposition \ref{prop:Obstruction-diagonal} is then completed by Lemma \ref{lem:HPT-class} below, where we note that $Z_0$ is defined over $k_0$ and so the function field  $\kappa(Z_0)$ of $Z_0$ may also be written as $k_0(Z_0)$.  
\end{proof}

\begin{lemma} 
\label{lem:HPT-class}
Let $
Z_0 = \{ w^2 - f_0 = 0 \} \subset \CP_{k_0} (1^5 , 2 ) $
where $f_0$ is as in (\ref{eq:HPT}). 
Then the class $\delta_{Z_0}\in\CH_0 (Z_{0 , k_0 (Z_0)})$ is non-zero in the quotient
\begin{align} \label{eq:delta-HPT-in-CH}
0\neq \delta_{Z_0 }\in \frac{\CH_0( Z_{0 , k_0 (Z_0)})/2}{\CH_0(Z_{0 })/2} .
\end{align}
\end{lemma}

\begin{proof}
As noted in   \cite{HPT2},  $Z_0$ is birational to the $(2,2)$ hypersurface in $\CP^2 \times \CP^3$ as described in \cite{HPT}.
The lemma follows thus from \cite[Proposition 11]{HPT} and \cite[Theorem 9.2]{Sch-JAMS} by similar arguments as in \cite[Proposition 3.1 and 7.1]{Sch-JAMS}.
We give some details for convenience of the reader.

Recall that the double cover $Z_0\to \CP^4$ is given by the equation
$$
w^2=z_0 z_1 z_3^2 + z_0 z_2 z_4^2 + z_1 z_2 (z_0^2 + z_1^2 + z_2^2 - 2( z_0 z_1 + z_0 z_2 + z_1 z_2)) .
$$
The branch locus in $\CP^4$ has multiplicity two along the line $\ell:=\{z_0=z_1=z_2=0 \}$.
Let $Z_0'=Bl_{\ell} Z_0$ and let $\pi: Z_0'\to \CP^2$ denote the morphism induced by projection from $\ell$.
The generic fibre of $\pi$ is the quadric surface from \cite{HPT} and so  $\alpha=(z_1/z_0,z_2/z_0)\in H^2(k_0(\CP^2),\Z/2)$ satisfies
$$
0\neq \pi^\ast \alpha\in H^2_{nr}(k_0(Z_0')/k_0,\Z/2) ,
$$
see \cite[Proposition 11]{HPT} (\textit{loc. cit.} is written over the field of complex numbers, but the same arguments work over any algebraically closed field of characteristic different from two, cf.\ \cite[Example 4.2]{Sch-JAMS}).
The exceptional divisor $E$ of $Bl_{\ell}Z_0\to Z_0$ is given by the equation
$$
w^2=z_0 z_1 z_3^2 + z_0 z_2 z_4^2 .
$$
This is a conic bundle over $\CP^2$ whose generic fibre is the conic that corresponds to the symbol $\alpha$ and so $\pi^\ast \alpha$ vanishes when restricted to the generic point of $E$.
If $e\in E$ is a regular point, then the pullback map $H^2(\Spec \mathcal O_{E,e},\Z/2)\to H^2(k(E),\Z/2)$ is injective (see e.g.\ \cite[Theorem 3.6(a)]{Sch-survey}) and so the restriction of $\pi^\ast \alpha$
to any regular point of $E$ vanishes.
In particular, $\pi^\ast \alpha$ vanishes at any point of the generic fibre of $E\to \CP^2$.
Let now $\tau':Z_0''\to Z_0'$ be an alteration whose degree is odd, which exists by Gabber's theorem,  see \cite{IT}.
Since $k_0$ is perfect, $Z_0''$ is smooth.

Consider the composition $\tau:Z_0''\to Z_0$ of $\tau'$ with the blow-down map $Z_0'\to Z_0$.
Clearly,  $\tau$ is an alteration of odd degree.
The base change of $\tau$ to $k_0(Z_0)$ is an alteartion of odd degree of $Z_{0,k_0(Z_0)}$ that we denote with the same symbol.
For a contradiction, we assume that there is a class $z_1\in \CH_0(Z_{0,k_0(Z_0)})$ and a class $z'\in \CH_0(Z_0)$ such that
$$
\delta_{Z_0}=2z_1+z'_{k_0(Z_0)}\in \CH_0( Z_{0 , k_0 (Z_0)}).
$$
We restrict this class to the smooth locus of $Z_{0,k_0(Z_0)}$ and pull that back via $f$.
The localization exact sequence \cite[Proposition 1.8]{fulton} then shows that
$$
\delta_{\tau}=z_2+2z'_1+z''_{k_0(Z_0)}\in \CH_0( Z''_{0 , k_0 (Z_0)}),
$$
where $\delta_{\tau}$ is the zero cycle on $Z''_{0 , k_0 (Z_0)}$ induced by the graph of $\tau$, $z_2$ is supported on $\tau^{-1}(Z^{\sing}_{0,k_0(Z_0)})$ and $z'_1\in \CH_0(Z''_{0,k_0(Z_0)})$  and $z''\in \CH_0(Z''_0)$ are some classes.

We aim to compute the Merkurjev pairing (see e.g.\ \cite[\S 5]{Sch-survey}) of the above zero-cycle with the unramified cohomology class $\tau^\ast \alpha\in H^2_{nr}(k_0(Z_0'')/k_0,\Z/2)$.
Here we find
$$
\langle \delta_{\tau}, \tau^\ast \pi^\ast \alpha \rangle=\deg(\tau)\cdot  \pi^\ast \alpha \in H^2(k_0(Z_0),\Z/2).
$$
This class is nonzero, because $\deg(\tau)$ is odd and $\pi^\ast \alpha$ is nonzero.

On the other hand, 
$$
\langle z_2+2z'_1+z''_{k_0(Z_0)},  \tau^\ast \pi^\ast \alpha  \rangle=\langle z_2 +z''_{k_0(Z_0)},  \tau^\ast \pi^\ast \alpha  \rangle=\langle z_2  ,  \tau^\ast \pi^\ast \alpha  \rangle,
$$
where we used that $ \langle z''_{k_0(Z_0)},  \tau^\ast \pi^\ast \alpha \rangle=0$ as $ \tau^\ast \pi^\ast \alpha$ restricts to zero on any closed point of $Z_0$,  because $k_0$ is algebraically closed. 
We claim 
$$
\langle z_2  ,  \tau^\ast \pi^\ast \alpha  \rangle=0
$$
and it suffices to show our claim in the case where $z_2$ is a single point.
If $z_2$ does not map to the generic point of $\CP^2$ via the composition $f:=\pi\circ \tau:Z_0''\to \CP^2$, then the claim follows from \cite[Theorem 9.2]{Sch-JAMS}.
Otherwise, since $z_2$ is supported on $\tau^{-1}(Z^{\sing}_{0,k_0(Z_0)})$ and the generic fibre of  $\pi:Z_0'\to \CP^2$ is smooth, $\tau(z_2)$ is a point on the generic fibre of $E\to \CP^2$.
We have shown above that $\alpha$ vanishes when restricted to any point on the generic fibre of $E\to \CP^2$ and this implies that
$$
\langle z_2  ,  \tau^\ast \pi^\ast \alpha  \rangle=0 ,
$$
as we want.
Altogether we have thus shown that 
$$
0\neq \deg(\tau)\cdot  \pi^\ast \alpha=\langle \delta_{\tau}, \tau^\ast \pi^\ast \alpha \rangle=\langle z_2+2z'_1+z''_{k_0(Z_0)},  \tau^\ast \pi^\ast \alpha  \rangle=0 \in H^2(k_0(Z_0),\Z/2).
$$
This is a contradiction, which concludes the proof of the lemma.
\end{proof}

\section*{Acknowledgements}   
Thanks to John  Ottem for conversations and comments.
We are also grateful to a referee for comments that helped to improve the exposition.
 	This project has received funding from the European Research Council (ERC) under the European Union's Horizon 2020 research and innovation programme under grant agreement No 948066.


\begin{thebibliography}{BCTSS85} 


\bibitem[AM72]{artin-mumford}
M.\ Artin and D.\ Mumford, {\em Some elementary examples of unirational varieties which are not rational}, Proc.\ London Math.\ Soc.\ (3) \textbf{25} (1972), 75--95.


\bibitem[BCTSS85]{BCTSS}
A.\ Beauville, J.-T.\ Colliot-Th\'el\`ene, J.-J.\ Sansuc and P.\ Swinnerton-Dyer, {\em Vari\'et\'es Stablement Rationnelles Non Rationnelles},
Ann.\ Math.\ \textbf{121} (1985),  283--318.


\bibitem[BBG19a]{BBGa}
C.\ B\"ohning, H.-C.\ von Bothmer and M.\ van Garrel, {\em Prelog Chow groups of self-products of degenerations of cubic threefolds}, European Journal of Mathematics \textbf{8.1} (2022), 260--290.

\bibitem[BBG19b]{BBGb}
C.\ B\"ohning, H.-C.\ von Bothmer and M.\ van Garrel, {\em Prelog Chow rings and degenerations}, Rendiconti del Circolo Matematico di Palermo, Series 2 (2022), 1--34.

\bibitem[Bou06]{bourbaki}
N.\  Bourbaki,  {\em  Alg\`ebre commutative,  Chapitres 8 et 9,} Springer (2006).

\bibitem[CG72]{CG}
C.H.\ Clemens and P.A.\ Griffiths, {\em The intermediate Jacobian of the cubic threefold}, Ann.\ Math.\ \textbf{95} (1972), 281--356.


\bibitem[CL17]{CL}
A.\ Chatzistamatiou and M.\ Levine, {\em Torsion orders of complete intersections},  Algebra \& Number Theory \textbf{11} (2017), 1779--1835. 


\bibitem[CTP16]{CTP}
J.-L.\ Colliot-Th\'el\`ene and A.\ Pirutka, {\em Hypersurfaces quartiques de dimension 3 \ : non rationalit\'e stable},  Annales Sc.\ \'Ec.\ Norm.\ Sup.\ \textbf{49} (2016), 371--397.


\bibitem[CM98]{conte-murre}
A.\ Conte and J.\ Murre, {\em On a theorem of Morin on unirationality of the quartic fivefold}, Atti Accad.\ Sci.\ Torino Cl.\ Sci.\ Fis.\ Mat.\ Natur. \textbf{132} (1998), 49--59.



\bibitem[EGAIV.4]{EGAIV.4}
A.\ Grothendieck, {\em  \'El\'ements de g\'eom\'etrie alg\'ebrique IV. \'Etude locale des sch\'emas et des morphismes de sch\'emas, Quatri\`eme partie},   Publications Math\'ematiques de l'IH\'ES \textbf{32} (1967),   5--361.
 

\bibitem[Ful75]{fulton2}
W.\ Fulton, {\em Rational equivalence on singular varieties}, Publ.\ Math.\ IHES \textbf{45} (1975), 147--167.

\bibitem[Ful98]{fulton}
W.\ Fulton, {\em Intersection theory}, Springer--Verlag, 1998.

\bibitem[Har01]{hartl}
U.\ T.\ Hartl, {\em Semi-stability and base change}, Arch. Math. \textbf{77} (2001), 215-–221.


\bibitem[HPT18]{HPT}
B.\ Hassett, A.\ Pirutka and Yu.\ Tschinkel, {\em Stable rationality of quadric surface bundles over surfaces}, Acta Math.\ \textbf{220} (2018),  341--365.
 
\bibitem[HPT19]{HPT2}
B.\ Hassett, A.\ Pirutka and Yu.\ Tschinkel, {\em A very general quartic double fourfold is not stably rational}, Algebraic Geometry \textbf{6} (2019), 64--75.


\bibitem[IT14]{IT}
L.\ Illusie and M.\ Temkin,
{\em Expos\'e X. Gabber's modification theorem (log smooth case)}, Ast\'erisque (2014),
no.\ 363--364, 167--212, Travaux de Gabber sur l'uniformisation locale et la cohomologie
\'etale des sch\'emas quasi-excellents.


\bibitem[IM72]{IM}
V.A.\ Iskovskikh and Y.I.\ Manin, {\em Three-dimensional quartics and counterexamples to the L\"uroth problem}, Mat.\ Sb.\ \textbf{86} (1971), 140--166. Eng.\ trans., Math.\ Sb.\ \textbf{15} (1972), 141--166.
 

\bibitem[Kol95]{kollar}
J.\ Koll\'ar, {\em Nonrational hypersurfaces}, J.\ Amer.\ Math.\ Soc.\ \textbf{8} (1995), 241--249. 


\bibitem[KT19]{KT}
M.\ Kontsevich and Yu.\ Tschinkel, {\em Specialization of birational types}, Invent.\ Math.\ \textbf{217} (2019), 415--432. 



\bibitem[Lev05]{levine}
M.\ Levine, {\em Mixed Motives},
in: Handbook of K-theory, vol.\ \textbf{1},
E.M. Friedlander, D.R. Grayson, eds.,
429--522.  Springer--Verlag, 2005.

\bibitem[Mer08]{merkurjev}
A.S.\ Merkurjev, {\em Unramified elements in cycle modules}, J.\ London Math.\ Soc.\ \textbf{78} (2008), 51--64.


\bibitem[Mur73]{Mur}
J.P.\ Murre, {\em Reduction of the proof of the non-rationality of a non-singular cubic
threefold to a result of Mumford}, Comp.\ Math.\
\textbf{27} (1973), 63--8.


\bibitem[NO19]{NO}
J.\ Nicaise and J.C.\ Ottem, {\em Tropical degenerations and stable rationality},  to appear in Duke Math.\ Journal.

\bibitem[NS19]{NS}
J.\ Nicaise and E.\ Shinder, {\em The motivic nearby fiber and
degeneration of stable rationality},  Invent.\ Math.\ \textbf{217} (2019), 377--413.


\bibitem[Rob70]{Roberts} 
J.\ Roberts, {\em Chow's Moving Lemma}, 
Algebraic Geometry, Proceedings of the 5th Nordic Summer-School in Mathematics,
89--96, Oslo (1970), Wolters-Noordhoff, Groningen.


\bibitem[Sch19a]{Sch-Duke}
S.\ Schreieder, {\em On the rationality problem for quadric bundles}, Duke Math.\ J.\ \textbf{168} (2019), 187--223.

\bibitem[Sch19b]{Sch-JAMS}
S.\ Schreieder, {\em Stably irrational hypersurfaces of small slopes},
J.\ Amer.\ Math.\ Soc.\ \textbf{32} (2019), 1171--1199.


\bibitem[Sch21a]{Sch-torsion}
S.\ Schreieder, {\em Torsion orders of Fano hypersurfaces},  Algebra Number Theory \textbf{15} (2021), 241--270.


\bibitem[Sch21b]{Sch-survey}
S.\ Schreieder, {\em Unramified cohomology, algebraic cycles and rationality},  in: G. Farkas et al. (eds), Rationality of Varieties,
Progress in Mathematics \textbf{342}, Birkh\"auser (2021), 345--388.


\bibitem[Ska22]{skauli}
B.\ Skauli, {\em A (2,3)-complete intersection fourfold with no decomposition of the diagonal}, manuscripta math.\ (2022),
\url{https://doi.org/10.1007/s00229-022-01386-y}.


\bibitem[Tot16]{totaro}
B.\ Totaro, {\em Hypersurfaces that are not stably rational}, J.\ Amer.\ Math.\ Soc.\ \textbf{29} (2016), 883--891. 
 

\bibitem[Voi15]{voisin}
C.\ Voisin, {\em Unirational threefolds with no universal codimension 2 cycle}, Invent.\ Math.\ \textbf{201} (2015), 207--237.



\end{thebibliography}
\end{document}